\documentclass[a4paper, 10pt, DIV10, headinclude=false, footinclude=false]{scrartcl}

\usepackage{amsthm, amsmath, amssymb, amsfonts,esint}
\usepackage[utf8]{inputenc}
\usepackage[english]{babel}
\usepackage{url}
\usepackage{mathtools}

\usepackage{tikz}
\usepackage{tikz-3dplot}
\tdplotsetmaincoords{70}{110}
\usepackage{pgfplots}
\usepackage{pgfplotstable}
\usepackage{booktabs}

\numberwithin{figure}{section}
\numberwithin{table}{section}
\numberwithin{equation}{section}

\newenvironment{abstr}[1]{ \vspace{.05in}\footnotesize
	\parindent .2in
	{\upshape\bfseries #1. }\ignorespaces}{\par\vspace{.1in}}
\newenvironment{Abstract}{\begin{abstr}{Abstract}}{\end{abstr}}
\newenvironment{keywords}{\begin{abstr}{Key words}}{\end{abstr}}
\newenvironment{AMS}{\begin{abstr}{AMS subject classifications}}{\end{abstr}}

\newtheorem{theorem}{Theorem}[section]
\newtheorem{lemma}[theorem]{Lemma}
\newtheorem{corollary}[theorem]{Corollary}
\newtheorem{proposition}[theorem]{Proposition}
\newtheorem{assumption}[theorem]{Assumption}

\theoremstyle{definition}

\newtheorem{example}[theorem]{Example}
\newtheorem{remark}[theorem]{Remark}

\DeclareMathOperator{\diam}{diam}
\newcommand{\nz}{\mathbb{N}}       
\newcommand{\rz}{\mathbb{R}}       
\newcommand\VA{\mathbf{A}}

\newcommand\CA{\mathcal{A}}
\newcommand\CB{\mathcal{B}}

\newcommand\CQ{\mathcal{Q}}

\newcommand\CS{\mathcal{S}}
\newcommand\CT{\mathcal{T}}

\newcommand{\UN}{\textup{N}}

\usepackage{todonotes}

\allowdisplaybreaks[4]

\begin{document}
	
	\title{Numerical homogenization for nonlinear strongly monotone problems%
	\thanks{Major parts of this work were carried out while the author was affiliated with the University of Augsburg. Further, the work conducted at KIT was funded by the Deutsche Forschungsgemeinschaft (DFG, German Research Foundation) -- Project-ID 258734477 -- SFB 1173 and by the Federal Ministry of Education and Research (BMBF) and the Baden-W\"urttemberg Ministry of Science as part of the Excellence Strategy of the German Federal and State Governments.}
	}
	\author{Barbara Verf\"urth\footnotemark[2]}
	\date{}
	\maketitle
	
	\renewcommand{\thefootnote}{\fnsymbol{footnote}}
	\footnotetext[2]{Institut für Angewandte und Numerische Mathematik, Karlsruher Institut f\"ur Technologie, Englerstr.~2, D-76131 Karlsruhe}
	\renewcommand{\thefootnote}{\arabic{footnote}}
	
	\begin{Abstract}
	In this work we introduce and analyze a new multiscale method for strongly nonlinear monotone equations in the spirit of the Localized Orthogonal Decomposition.
	A problem-adapted multiscale space is constructed by solving linear local  fine-scale problems which is then used in a generalized finite element method.
	The linearity of the fine-scale problems allows their localization and, moreover, makes the method very efficient to use.
	The new method gives optimal a priori error estimates up to linearization errors. The results neither require structural assumptions on the coefficient such as periodicity or scale separation nor higher regularity of the solution. 
	The effect of different linearization strategies is discussed in theory and practice.
	Several numerical examples including stationary Richards equation confirm the theory and underline the applicability of the method.
	\end{Abstract}
	
	\begin{keywords}
	multiscale method; numerical homogenization; nonlinear monotone problem;
	a priori error estimates
	\end{keywords}

	\begin{AMS}
	65N15, 65N30, 35J60, 74Q15	
	\end{AMS}

	\section{Introduction}
\label{sec:introduction}
Linear constitutive laws like Hooke's law in mechanics, Ohm's law in electromagnetics, or Darcy's law in fluid flow are very popular, but they are often not accurate enough in practical applications, for instance for high intensities.
Instead, nonlinear effects in the constitutive laws have to be taken into account which are often experimentally found and determined, see \cite{Zeidler} for a general overview.
In this article, we consider as model problem the following nonlinear monotone elliptic equation 
\[-\nabla \cdot \bigl( A(x, \nabla u)\bigr) =f,\]
where the exact assumptions as well as boundary conditions are specified later.
It is a representative model problem for quasilinear partial differential equations (PDEs) as they occur in mean curvature flow or for non-Newtonian fluids.
The transition from linear to nonlinear problems comes with huge additional challenges for the numerical treatment and analysis. As an illustrating example we mention optimal order $L^2$-estimates for the finite element method: The classical Aubin-Nitsche trick for linear problems is not applicable, so that, for a long time, only optimal order estimates in the energy norm \cite{Ciarlet} were known, see \cite{AV12quasilin} and the discussion therein. A similar observation applies to the effect of numerical integration, see~\cite{FZ87femnonlinear}.

With the view on practical applications such as fluid flow or elasticity, we do not only have to consider nonlinear constitutive laws as discussed above, but also have to consider (spatial) multiscale features in the material coefficients (here, in $A$). For instance, a fluid such as groundwater flows over large distances, while the properties of the soil changes over small distances, see, e.g.,~\cite{Richards}.
Hence, for applications such as the (quasilinear) porous medium equation, $A$ is subject to rapid variations and/or discontinuities on fine spatial scales or  even a cascade of (non-separable) scales. 
This coincidence of multiscale features and nonlinear material laws makes the problem intractable for standard methods.
For example, the finite element method \cite{AH16nonlinelliptic,Ciarlet,FZ87femnonlinear} will only give optimal convergence in the asymptotic regime, i.e., if the mesh resolves all features and scales present, which is prohibitively expensive even with today's computational resources.

In the case of spatially periodic $A$ (with period $\varepsilon\ll 1$), homogenization results using two-scale convergence \cite{All92twoscale,LNW02twoscale} prove that the solutions of the above model problem converge to the solution of an again monotone elliptic (homogenized) problem for $\varepsilon\to 0$ .
The nonlinear effective diffusion tensor can be computed by solving nonlinear so-called cell problems.
The (finite element) heterogeneous multiscale method is inspired by this analytical process and it is studied successfully for nonlinear problems in a series of papers \cite{ABV15rbhmmquasilinelliptic,AH16fehmmnonlinparabolic,AHV15nonlinparaboliclinearized,AV14fehmmquasilinelliptic,Henn11phdhmm,HO15hmmmonotone}.
In most cases, the macroscopic nonlinear form involves nonlinear reconstruction operators which require the solution of nonlinear cell problems at each macroscopic quadrature point to incorporate fine-scale information.
For parabolic equations, \cite{AHV15nonlinparaboliclinearized} linearizes the macroscopic and cell computations using information from the previous time step.
The sparse multiscale FEM \cite{Hoa08sparsefemnonlinear}  tries to reduce the complexity of solving cell problems and a homogenized equation by the introduction of sparse approximations.
Another idea is to modify or enrich the standard finite element basis by problem-adapted functions. This is used in the (generalized) multiscale finite element method, for which nonlinear problems are discussed in \cite{CESY17gmsfemnonlinear,EGLP14gmsfemnonlin,EHG04msfem}.
Again nonlinear problems have to be solved locally to construct the problem-adapted functions.

The main contribution of this article is the introduction of a new multiscale method for nonlinear strongly monotone problems and its numerical analysis. The idea is to construct a multiscale  space by solving local fine-scale problems in the spirit of the Localized Orthogonal Decomposition (LOD) \cite{MP14LOD,Pet15LODreview}.
In contrast to the above discussed methods, the basis construction only requires the solution of \emph{linear} problems and hence is embarrassingly easy.
Moreover, this linearization idea drastically reduces the computational effort for generating a problem-dependent basis and thereby provides a conceptually new view on the treatment of nonlinear multiscale problems.
We derive optimal convergence rates (with respect to the mesh size $H$) up to linearization errors without any assumption on the regularity of the exact solution or special properties such as periodicity or scale separation for the coefficient.
The occurring linearization errors and resulting possible choices of the linearization are discussed and compared.
Extensive numerical experiments show the good performance of the method in agreement with the theoretical estimates. 
We study periodic as well as completely random multiscale coefficients and also include a model for stationary Richards equation with a high contrast channel.
Besides several linear problem classes, the LOD has already been studied for semilinear equations \cite{HMP14lodsemilinear} and a nonlinear eigenvalue problem related to the Gross-Pitaevskii equation \cite{HMP14lodgrosspit}.
These problems, however, are only semilinear and can therefore be handled easier. Yet, we emphasize that these previous works can be re-interpreted in the current framework.
We mention the close connections of the LOD to (analytical) homogenization \cite{GP17lodhom,PVV19homloddomaindecomp}, domain decomposition iterative solvers \cite{KPY16LODiterative,KY16LODiterative,PVV19homloddomaindecomp}, and so-called gamblets \cite{Owhadi2017,OZ11gfem}. 
Hence, the current approach can give interesting and useful insights in these areas for nonlinear problems in the future as well.

The article is organized as follows: Section \ref{sec:setting} introduces the setting and the standard finite element discretization.
We introduce the multiscale method including linearization and localization in Section~\ref{sec:lod}. The arising errors are analyzed in Section~\ref{sec:error}. Finally, we present several numerical experiments confirming our theory and showing possible applications in Section \ref{sec:numexp}.

\section{Problem formulation and discretization}
\label{sec:setting}
In this section we formulate the considered model problem and introduce necessary finite element prerequisites.
We use standard notation on Sobolev spaces.
Throughout the whole article, let $\Omega\subset \mathbb{R}^d$ be a bounded Lipschitz domain.
For a subdomain $D\subset \Omega$,
let $\|\cdot \|_{0, D}$, $\|\cdot\|_{1, D}$, and $|\cdot|_{1, D}$ denote the standard $L^2(D)$-norm, $H^1(D)$-norm, and $H^1(D)$-semi norm, respectively. 
Furthermore, $(\cdot, \cdot)_D$ denotes the standard $L^2$ scalar product on $D$.
We will omit the subscript $D$ if it equals the full domain $\Omega$.

\subsection{Model problem}
\label{subsec:problems}
We consider the following nonlinear elliptic problem: Find $u:\Omega\to \mathbb{R}$ such that
\begin{equation}
	\label{eq:nonlinerastrong}
	\begin{aligned}
		-\nabla \cdot \bigl( A(x, \nabla u)\bigr) &=f &&\text{in }\Omega,\\
		u&=0&&\text{on }\partial \Omega
	\end{aligned}
\end{equation}
with a right-hand side $f\in L^2(\Omega)$.
The corresponding weak formulation, with which we will work in the following, reads: Find $u\in H^1_0(\Omega)$ such that
\begin{equation}
	\label{eq:nonlinearweak}
	\CB(u;v):=\bigl(A(x, \nabla u), \nabla v\bigr)_\Omega = (f, v)_\Omega \qquad \text{for all}\quad v\in H^1_0(D).
\end{equation}
For simplicity, we restrict ourselves to homogeneous Dirichlet boundary conditions, but non-homo\-ge\-neous and Neumann boundary conditions could be treated as well, see~\cite{HM14LODbdry}.
Moreover, we focus on nonlinearities in the highest derivative only, additional (nonlinear) low-order terms can easily be handled as well, cf.~\cite{HMP14lodsemilinear}.
We now specify our assumptions on $A$.
\begin{assumption}
	\label{asspt:monotone}
	The nonlinearity $A:\Omega\times \mathbb{R}^d\to \mathbb{R}^d$ satisfies
	\begin{enumerate}
		\item $A(\cdot, \xi)\in L^\infty(\Omega;\mathbb{R}^d)$ for all $\xi\in\rz^d$ and $A(x, \cdot)\in C^1(\rz^d;\rz^d)$ for almost every $x\in \Omega$;
		\item  there is $\Lambda>0$ such that $|A(x, \xi_1)-A(x,\xi_2)|\leq \Lambda |\xi_1-\xi_2|$ for almost every $x\in \Omega$ and all $\xi_1, \xi_2\in \rz^d$;
		\item there is $\lambda>0$ such that $\bigl(A(x, \xi_1)-A(x, \xi_2)\bigr)\cdot (\xi_1-\xi_2)\geq \lambda |\xi_1-\xi_2|^2$ for almost every $x\in \Omega$ and all $\xi_1, \xi_2\in \rz^d$;
		\item $|A(x, 0)|\leq C_0$ for almost every $x\in \Omega$.
	\end{enumerate}
\end{assumption}

Assumption \ref{asspt:monotone} implies that 
\[|\CB(v_1;\psi)-\CB(v_2;\psi)|\leq \Lambda |v_1-v_2|_1\, |\psi|_1\qquad\text{and}\qquad \CB(v;v-\psi)-\CB(\psi;v-\psi)\geq \lambda|v-\psi|^2_1\]
for all $v, v_1, v_2,\psi\in H^1_0(\Omega)$.
Therefore, the model problem \eqref{eq:nonlinearweak} has a unique solution $u\in H^1_0(\Omega)$, which satisfies
\begin{equation}
	\label{eq:stabu}
	|u|_1\leq \Lambda/\lambda \|f\|_0,
\end{equation}
see~\cite[Chapter 5]{Ciarlet}.
As discussed in the introduction, we implicitly assume that $A$ is subject to rapid oscillations or discontinuities on a rather fine scale with respect to the spatial variable $x$.

We write $a\lesssim b$ in short for $a\leq C b$ with a constant C independent of the mesh size $H$ and the oversampling parameter $m$ introduced later. 
However, $C$ may depend on the monotonicity and Lipschitz constants $\lambda, \Lambda$ of $A$ (cf.\ Assumption \ref{asspt:monotone}).

\subsection{Finite element discretizations}
\label{subsec:fem}
We cover $\Omega$ with a regular mesh $\CT_H$ consisting of simplices; however, a mesh with quadrilaterals would equally be possible.
The mesh is assumed to be shape regular in the sense that the aspect ratio of the elements of $\CT_H$ is bounded uniformly from below.
We introduce the mesh size $H=\max_{T\in \CT_H}\diam T$ and assume that this is rather coarse, in particular, $\CT_H$ does not resolve the possible heterogeneities in $A$.
We discretize the space $H^1_0(\Omega)$ with the lowest order Lagrange elements over $\CT_H$, and denote this space by $V_H$.
This means that $V_H=H^1_0(\Omega)\cap \CS^1(\CT_H)$, where $\CS^1(\CT_H)$ denotes the space of element-wise polynomials of total degree $\leq 1$.

The standard finite element method now seeks a (discrete) solution $u_H\in V_H$ such that
\[\CB(u_H; v_H)=(f,v_H)_\Omega \qquad \text{for all}\quad v_H\in V_H.\]
This results in a nonlinear system which can be (approximatively) solved via an iteration such as Newton's method.
It is well-known that  the properties of $A$ and Galerkin orthogonality imply
\begin{equation}
	\label{eq:quasiopt}
	|u-u_H|_1\lesssim \inf_{v_H\in V_H}|u-v_H|_1,
\end{equation}
see~\cite[Chapter 5]{Ciarlet}.
This quasi-optimality by the way holds for any conforming subset $\tilde{V}_H\subset H^1_0(\Omega)$.
For the standard finite element method (with linear elements) it is furthermore well-known to have the following error estimates 
\[\|u-u_H\|_1\leq C H \|u\|_{H^{2}(\Omega)}\qquad \text{and}\qquad \|u-u_H\|_0\leq C H^{2}\|u\|_{H^{2}(\Omega)},\]
see~\cite{AH16nonlinelliptic}.
The higher regularity ($u\in H^2(\Omega)$) of the exact solution required in those estimates may not be attainable for nonlinearities $A$ with spatial discontinuities.
Even if $u\in H^2(\Omega)$, the corresponding norm depends on spatial derivatives of $A$ which behave like $\varepsilon^{-1}$ for coefficients varying on a scale $\varepsilon$.
In practice this implies that $H$ needs to be at least $\varepsilon$ in order to observe the linear convergence in the $H^1(\Omega)$-norm.
In other words, for small $\varepsilon$, there is a large pre-asymptotic region where the error stagnates (at a high level) in practice.

The goal of the multiscale method presented in Section \ref{sec:lod} is to circumvent both issues (higher regularity of the solution and dependence on the variations of $A$).
At the heart of the method is the choice of a suitable interpolation operator and we now introduce the required properties as well as an appropriate example.
Let $I_H: H^1_0(\Omega)\to V_H$ denote a bounded local linear projection operator, i.e., $I_H\circ I_H = I_H$, with the following stability and approximation properties for all $v\in H^1_0(\Omega)$
\begin{align}
	\label{eq:IHstab}
	|I_H v|_{1, T}&\lesssim |v|_{1, \UN(T)},\\
	\label{eq:IHapprox}
	\|v-I_H v\|_{0, T}&\lesssim H|v|_{1, \UN(T)}.
\end{align}
where the constants are independent of $H$ and $\UN(T):=\{K\in \CT_H: K\cap T\neq \emptyset\}$ denotes the neighborhood of an element.
A possible choice (which we use in our implementation of the method)
is to define $I_H:=E_H\circ\Pi_{H}$. $\Pi_{H}$ is the $L^2$-projection onto the elementwise affine functions $\mathcal{S}^1(\mathcal{T}_H)$, and $E_H$ is the averaging operator that maps discontinuous functions in $\CS^1(\CT_H)$  to $V_H$ by
assigning to each free vertex the arithmetic mean of the corresponding
function values of the neighboring cells, that is, for any $v\in \CS^1(\CT_H)$
and any vertex $z$ of $\CT_H$,
\begin{equation*}
	(E_H(v))(z) =
	\sum_{T\in\CT_H,\;z\in T}v|_T (z) 
	\bigg/
	\operatorname{card}\{K\in\CT_H\,,\,z\in K\}.
\end{equation*}
For further details on suitable interpolation operators we refer to \cite{EHMP16LODimpl}.

\section{Computational multiscale method}
\label{sec:lod}
In the following, we assume that an interpolation operator $I_H:H^1_0(\Omega)\to V_H$ satisfying the projection property as well as \eqref{eq:IHstab} and \eqref{eq:IHapprox} is at hand.
Abbreviating $W:=\ker I_H$, we have the splitting $H^1_0(\Omega)=V_H\oplus W$. The main idea of the Localized Orthogonal Decomposition \cite{MP14LOD,Pet15LODreview} is to make this splitting problem-dependent. In the linear elliptic case the splitting is orthogonalized with respect to the energy scalar product.
Below, we discuss how this idea can be transferred to the nonlinear case. We introduce a linearization procedure in the next subsection which makes the computation of a multiscale space in the spirit of the LOD possible. 
Afterwards, we present the localized computation of the new multiscale basis functions.

\subsection{An ideal method and its linearization}
\label{subsec:ideallin}
Motivated by linear elliptic equations, one could (naively) try to introduce a Galerkin method over a subset $V_H^{\mathrm{nl}, \mathrm{ms}}\subset H^1_0(\Omega)$, i.e., we seek $u_H^{\mathrm{nl}, \mathrm{ms}}$ such that
\[\CB(u_H^{\mathrm{nl}, \mathrm{ms}}; v)=(f, v)_\Omega \qquad \text{for all}\quad v\in V_H^{\mathrm{nl}, \mathrm{ms}},\]
where the set $V_H^{\mathrm{nl}, \mathrm{ms}}$ is defined via
\begin{equation}
	\label{eq:nonlinearortho}
	\CB(v_H^{\mathrm{nl}, \mathrm{ms}}; w)=0\qquad \text{for all}\quad v_H^{\mathrm{nl}, \mathrm{ms}}\in V_H^{\mathrm{nl}, \mathrm{ms}}\quad \text{and all}\quad w\in W.
\end{equation}
This is the orthogonalization idea behind the original method, see~\cite{MP14LOD,Pet15LODreview}.
Due to the quasi-optimality \eqref{eq:quasiopt} and the properties \eqref{eq:IHstab} and \eqref{eq:IHapprox} of $I_H$, one obtains the a priori error estimate
\[|u-u_H^{\mathrm{nl}, \mathrm{ms}}|_1\lesssim H \|f\|_0\]
with optimal rate in the mesh size, independent of the regularity of the continuous solution $u$.
This estimate is derived similar to the linear case  \cite{MP14LOD, Pet15LODreview}.
Because of the nonlinearity of $\CB$ in its first argument, however, $V_H^{\mathrm{nl}, \mathrm{ms}}$ is no longer a linear subspace.
To be more precise, it holds $V_H^{\mathrm{nl}, \mathrm{ms}}=(\operatorname{id}-\CQ^{\mathrm{nl}})V_H$, where $\CQ^{\mathrm{nl}}:V_H\to W$ solves
\begin{equation}
	\label{eq:nonlinearcorrec}
	\CB(v_H-\CQ^{\mathrm{nl}}v_H; w)=0\qquad \text{for all}\quad w\in W.
\end{equation}
Here, we clearly see that $\CQ^{\mathrm{nl}}$ is a nonlinear operator.
Therefore, it is by no means clear whether the proposed multiscale method is at all well defined. Even if this is the case, the method is very complicated as it involves two coupled nonlinear problems, where \eqref{eq:nonlinearcorrec} is additionally posed on the fine scale.

Here, we propose the following simple yet effective linearization approach.
We approximate the nonlinearity $A(x, \nabla v)$ by a function $A_L(x, \nabla u^*, \nabla v)$. Here, $A_L:\Omega\times \rz^d\times \rz^d\to \rz^d$ is affine in its last argument and we call $u^*\in H^1_0(\Omega)$ the linearization point. 
	We make the following assumption on $A_L$.
	\begin{assumption}
		\label{asspt:linearization}
		Write $A_L(x, \nabla u^*,\nabla v)=\VA_L(x,\nabla u^*)\nabla v+b_L(x,\nabla u^*)$  with $\VA_L(x,\nabla u^*)\in\rz^{d\times d}$ and $b_L(x,\nabla u^*)\in\rz^d$.
		We assume that
		\begin{itemize}
			\item $\VA_L(\cdot, \nabla u^*)\in L^\infty(\Omega; \rz^{d\times d})$ and $b_L(\cdot, \nabla u^*)\in L^2(\Omega; \rz^d)$ for all $u^*\in H^1_0(\Omega)$;
			\item $\VA_L(x,\nabla u^*)$ is symmetric for all $u^*\in H^1_0(\Omega)$; 
			\item there exists $0<\lambda_L\leq \Lambda_L$ such that
			\[\lambda_L|\xi|^2\leq \VA_L(x,\nabla u^*)\xi\cdot \xi\leq \Lambda_L|\xi|^2\qquad \text{for all}\quad x\in\Omega, u^*\in H^1_0(\Omega),\xi\in\rz^d.\]
		\end{itemize}
	\end{assumption}
The assumption of symmetry is only made for convenience and to avoid cluttering of notation in the following.
Although this linearization model may seem rather abstract, it  includes Newton-type as well as Ka\v{c}anov-type linearizations as illustrated in the example, cf.~\cite{EEV11nonlinearapost}.

\begin{example}\label{ex:linearization}
	Newton-type linearizations are based on a Taylor expansion up to the first order of the nonlinearity around the linearization point. 
	In particular, we approximate $A(x,\nabla v)\approx A(x,\nabla u^*)+D_\xi A(x, \nabla u^*)\nabla(v-u^*)$, where $D_\xi A$ denotes the Jacobian of $A$ with respect to the second argument.
	In the notation of Assumption \ref{asspt:linearization}, this means that $b_L(x, \nabla u^*)=A(x,\nabla u^*)-D_\xi A(x,\nabla u^*)\nabla u^*$ and $\boldsymbol{A}_L(x,\nabla u^*)=D_\xi A(x,\nabla u^*)$.
	The assumptions of strict monotonicity and Lipschitz continuity on $A$ (cf.\ Assumption \ref{asspt:monotone}) imply that $\boldsymbol{A}_L$ indeed satisfies Assumption \ref{asspt:linearization}, see~\cite[Lemma 6.5.2]{Hub15phdhmm}.
	
	In the case that $A$ takes the form $A(x,\nabla u)=\alpha(x,|\nabla u|^2)\nabla u$,  Ka\v{c}anov-type linearizations are very popular, which ``freeze the nonlinearity''.
	In the language of Assumption \ref{asspt:linearization}, one sets $A_L(x,\nabla u^*, \nabla v):=\alpha(x,|\nabla u^*|^2)\nabla v$, and, hence, $\boldsymbol{A}_L(x, \nabla u^*)=\alpha(x, |\nabla u^*|^2)$ and $b_L(x, \nabla u^*)=0$.
\end{example}

Let us come back to the linearization of \eqref{eq:nonlinearcorrec}. 
We pick and fix  a function $u^*\in H^1_0(\Omega)$ and set  $\mathfrak{A}:=\VA_L(\cdot, \nabla u^*)\in L^\infty(\Omega;\rz^{d\times d})$.
We define the \emph{linear} correction operator $\CQ: V_H\to W$ via
\begin{equation}
	\label{eq:correctionlinearized}
	\CA(v_H-\CQ v_H, w)=0\qquad \text{for all}\quad w\in W,
\end{equation}
where the bilinear form $\CA$ is defined as
\[\CA(v, \psi):=(\mathfrak{A}(x)\nabla v,\nabla \psi)_\Omega.\]
Due to Assumption \ref{asspt:linearization}, the linear corrector problem \eqref{eq:correctionlinearized} has a unique solution. 
Further, note that \eqref{eq:correctionlinearized} can equivalently be written as
\[\int_\Omega A_L(x, \nabla u^*, \nabla \CQ v_H)\cdot \nabla w\, dx=\int_\Omega A_L(x, \nabla u^*, \nabla v_H)\cdot \nabla w\, dx \qquad \text{for all}\quad w\in W.\]

Having linearized $\mathcal{Q}$, we define the linear multiscale space $V_H^{\mathrm{ms}}:=(\operatorname{id}-\CQ)V_H$.
This problem-adapted space is used in a Galerkin method to seek $u_H^{\mathrm{ms}}\in V_H^{\mathrm{ms}}$ as solution of the \emph{nonlinear} problem
\begin{equation}
	\label{eq:linlodgal}
	\CB(u_H^{\mathrm{ms}}; v_H^{\mathrm{ms}})=(f, v_H^{\mathrm{ms}})_\Omega\qquad \text{for all}\quad v_H^\mathrm{ms}\in V_H^{\mathrm{ms}}.
\end{equation}
Let $\{\lambda_z\}_z$ be the nodal basis of $V_H$ (i.e., the standard hat functions). Then $\{\lambda_z-\CQ\lambda_z\}_z$ forms a basis of $V_H^{\operatorname{ms}}$. Note that this requires only to solve linear problems.
After this basis has been (pre-)computed, the nonlinear problem \eqref{eq:linlodgal} can be solved with any suitable iteration method such as Newton's method, which is rather cheap since the dimension of the multiscale space is small (note $\operatorname{dim}(V_H^{\mathrm{ms}})=\operatorname{dim}(V_H)$).
We emphasize that although we linked nonlinear iterative methods and linearization strategies in Example~\ref{ex:linearization}, the iterative method chosen to solve \eqref{eq:linlodgal} does not need to correspond to the linearization strategy chosen for \eqref{eq:correctionlinearized}.

\subsection{Localization of the basis generation}
\label{subsec:local}
As in the linear case, the corrector problems \eqref{eq:correctionlinearized} are global fine-scale problems, which are as expensive to solve as the solution of a (linear) multiscale model problem on a fine-scale mesh.
However since \eqref{eq:correctionlinearized} is a standard elliptic problem (cf. Assumption \ref{asspt:linearization}), we can localize these corrector problems in the well-known way for the linear case.
To this end, we define the neighborhood
\[\UN(T)=\bigcup_{K\in \CT_H, T\cap K\neq \emptyset} K\]
associated with an element $T\in\CT_H$.
Thereby, for any $m\in\nz_0$, the $m$-layer patches are defined inductively via $\UN^{m+1}(T)=\UN(\UN^m(T))$ with $\UN^0(T):=T$.
The shape regularity implies that there is a bound $C_{\operatorname{ol}, m}$ (depending only on $m$) of the number of the elements in the $m$-layer patch, i.e.,
\begin{equation}
	\label{eq:Colm}
	\max_{T\in \CT_H}\operatorname{card}\{K\in \CT_H: K\subset \UN^m(T)\}\leq C_{\operatorname{ol}, m}.
\end{equation}
Throughout this article, we assume that $\CT_H$ is quasi-uniform, which implies that $C_{\operatorname{ol}, m}$ grows at most polynomially with $m$.

We then define the truncated correction operator $\CQ_m: V_H \to W$ as $\CQ_m=\sum_{T\in \CT_H}\CQ_{T, m}$, where for any $v_H\in V_H$ the truncated element corrector $\CQ_{T, m}v_H\in W(\UN^m(T)):=\{w\in W: w=0\text{ in }\Omega\setminus\UN^m(T)\}$ solves
\begin{equation}
	\label{eq:lincorrectorlocal}
	\CA_{\UN^m(T)}(\CQ_{T, m}v_H, w)=\CA_T(v_H, w)\qquad\text{for all}\quad w\in W(\UN^m(T)).
\end{equation}
Here, $\CA_D$ denotes the restriction of the bilinear form $\CA$ to the subdomain $D\subset \Omega$.
We then set up the multiscale space $V_{H,m}:=(\operatorname{id}-\CQ_m) V_H$.
For each element $T\in \CT_H$, we only have to solve $d$ problems of type \eqref{eq:lincorrectorlocal} with $v_H|_T=x_j$, $j=1,\ldots, d$, or precisely, the following cell problems: Find $q_{T,m}^{(j)}\in W(\UN^m(T))$, $j=1,\ldots, d$, such that
\[\int_{\UN^m(T)} \mathfrak{A}\nabla q_{T,m}^{(j)}\cdot \nabla w\, dx = \int_T\mathfrak{A}e_j\cdot \nabla w\, dx \qquad \text{for all}\quad w\in W(\UN^m(T)),\]
where $e_j$ denotes the $j$th canonical unit vector.
Denoting by $\{\lambda_z\}_z$ the standard hat functions, a basis of $V_{H,m}$ is hence given by 
\[\Bigl\{\lambda_z-\sum_{T\in\CT_H,\;z\in T}\sum_{j=1}^d\Bigl(\frac{\partial}{\partial x_j} \lambda_z|_T\Bigr)q_{T, m}^{(j)}\Bigr\}_z.\]

The localized multiscale method consists of replacing $V_H^{\mathrm{ms}}$ by $V_{H, m}$ in \eqref{eq:linlodgal}. More precisely, we seek (in a Galerkin method) $u_{H,m}\in V_{H, m}$ such that
\begin{equation}
	\label{eq:lodgal}
	\CB(u_{H,m};v_{H,m})=(f, v_{H,m})_\Omega\qquad\text{for all}\quad v_{H,m}\in V_{H,m}.
\end{equation}
Again, the nonlinear problem \eqref{eq:lodgal}  is solved with an iterative method, where the multiscale basis can be pre-computed.
This requires the storage of all correctors, which can be very memory consuming since $q_{T,m}^{(j)}$ includes fine-scale features.
Instead the correctors could also be computed on the fly inside each Newton iteration.
Note that $\CQ_m$ and therefore also the solution $u_{H,m}$ depends on $\mathfrak{A}$ and thereby, they implicitly depend on (i) the chosen linearization model and (ii) the chosen  linearization point $u^*$ in Section \ref{subsec:ideallin}.
The choice of $u^*$ and its consequences will be discussed in Section \ref{subsec:errorlinear} below.
\begin{remark}
	\label{rem:PGLOD}
	To avoid communication between the correctors, one can also consider the Petrov-Galerkin method to seek $u_{H, m}^{PG}\in V_H$ such that
	\begin{equation*}
		\CB(u_{H, m}^{PG}; v_{H,m})=(f, v_{H,m})_\Omega\qquad\text{for all}\quad v_{H,m}\in V_{H,m}.
	\end{equation*}
	In the Petrov-Galerkin method, $q_{T,j}$ and $q_{T^\prime, j}$ for $T, T^\prime\in \CT_H$ with $T\neq T^\prime$ are never needed at the same time. Hence, these correctors can immediately be discarded once the contributions of element $T$ to the linear system (in each Newton iteration) are assembled.
	Hence, when memory becomes the limiting factor, the Petrov-Galerkin variant is preferable, see~\cite{EGH15LODpetrovgalerkin,EHMP16LODimpl}.
	Note, however, that $u_{H,m}^{PG}$ only contains information on the coarse scale $H$.
	The following error analysis will be restricted to the Galerkin case for simplicity.
\end{remark}
\begin{remark}
	The present method is still semi-discrete since the corrector problems \eqref{eq:lincorrectorlocal} are infinite-dimensional. The discretization procedure for them is equivalent to the case of linear elliptic equations: We introduce a second (fine) simplicial mesh $\CT_h$ of $\Omega$ which resolves all features of $A$.
	Denoting by $V_h=H^1_0(\Omega)\cap\CS^1(\CT_h)$ the corresponding lowest order Lagrange finite element space, we set $W_h(\UN^m(T)):=\{w_h\in V_h: w_h=0 \text{ in }\Omega\setminus\UN^m(T), I_Hw_h=0\}$ and discretize \eqref{eq:lincorrectorlocal} by solving over the space $W_h(\UN^m(T))$ instead of $W(\UN^m(T))$.
	In the following, we work with the semi-discrete version and emphasize that similar error estimates (with respect to a reference solution $u_h\in V_h$) can be shown in the fully discrete variant, as illustrated for the linear case, see~\cite{HM14LODbdry,HP13oversampl,MP14LOD}. 
\end{remark}

\section{Error analysis}
\label{sec:error}
In this section, let the linear model $\mathfrak{A}$ be fixed, i.e., the linearization model $A_L$ and the linearization point $u^*\in H^1_0(\Omega)$ are fixed.
Occasionally, we will also use $\mathfrak{A}^u:=\VA_L(\cdot, \nabla u)$ with the above fixed linearization model $A_L$ and the exact solution $u$ to \eqref{eq:nonlinearweak}. $\CQ_m^u$ then denotes the corrector associated with $\mathfrak{A}^u$, i.e., $\CQ_m^u$ is defined via \eqref{eq:lincorrectorlocal} with $\mathfrak{A}$ replaced by $\mathfrak{A}^u$.
We have the following result on the error $\mathcal{Q}-\mathcal{Q}_m$.
\begin{proposition}\label{prop:errortrunc}
	Let Assumptions \ref{asspt:monotone} and  \ref{asspt:linearization} be fulfilled.
	Let $\CQ$ be  defined in \eqref{eq:correctionlinearized} and $\CQ_m$ defined via \eqref{eq:lincorrectorlocal}. 
	There exists $0<\beta<1$ such that for any $v_H\in V_H$
	\[|(\CQ-\CQ_m)v_H|_1\lesssim C_{\operatorname{ol}, m}^{1/2}\,\beta^m\,|v_H|_1.\]
\end{proposition}
Proposition \ref{prop:errortrunc} follows from the  linear elliptic case in \cite{HM14LODbdry,MP14LOD,Pet15LODreview}.
The main idea is that the ideal element corrector $\CQ_T$, which is defined via \eqref{eq:lincorrectorlocal} with $\UN^m(T)=\Omega$, decays exponentially fast (measured in $m$) away from $T$.
With a slightly different localization strategy, the procedure can also be interpreted in the spirit of an iterative domain decomposition solver, see~\cite{KPY16LODiterative,KY16LODiterative}.
Since Assumption \ref{asspt:linearization} holds for all $u^*\in H^1_0(\Omega)$, Proposition~\ref{prop:errortrunc} is still valid if we replace $\CQ$ and $\CQ_m$ by $\CQ^u$ and $\CQ_m^u$, respectively.

We first discuss estimates for the Galerkin method \eqref{eq:lodgal} in Section \ref{subsec:errorlod}.
(Additional) error terms arise from the linearization, which we discuss separately in Section \ref{subsec:errorlinear} together with the choice of $u^*$.

\subsection{A priori error estimates}
\label{subsec:errorlod}
Since $V_{H,m}$ is a linear subspace of $H^1_0(\Omega)$, the Galerkin method \eqref{eq:linlodgal} is automatically well defined, i.e., there exists a unique solution $u_{H,m}$ and its satisfies the following error estimate.
\begin{theorem}\label{thm:errorgal}
	Let Assumptions \ref{asspt:monotone} and  \ref{asspt:linearization} be fulfilled.
	Let $u$ be the solution to \eqref{eq:nonlinearweak} and $u_{H,m}$ the solution to \eqref{eq:lodgal}.
	Then it holds that
	\begin{equation}\label{eq:errorgalenergy}
		|u-u_{H,m}|_1\lesssim (H+C_{\operatorname{ol}, m}^{1/2}\,\beta^m)\|f\|_0+\eta_{\operatorname{lin}}((\operatorname{id}-\CQ)I_Hu)
	\end{equation}
	as well as
	\begin{equation}\label{eq:errorgalenergy1}
		|u-u_{H,m}|_1\lesssim (H+C_{\operatorname{ol}, m}^{1/2}\,\beta^m)\|f\|_0+\eta_{\operatorname{lin}}(u)
	\end{equation}
	with, for any $v\in H^1_0(\Omega)$, the linearization error
	\begin{align*}
		\eta_{\operatorname{lin}}(v)&:=\sup_{w\in W, |w|_1=1}\Bigl|\int_\Omega \bigl[ A(x, \nabla v)-A_L(x, \nabla u^*, \nabla v)\bigr]\cdot \nabla w\, dx\Bigr|.
	\end{align*}
	Further, if $A_L(x, \nabla u, \nabla u)=A(x, \nabla u)$, it also holds that
	\begin{equation}\label{eq:errorgalenergy2}
		|u-u_{H,m}|_1\lesssim (H+C_{\operatorname{ol}, m}^{1/2}\,\beta^m)\|f\|_0 + |(\CQ_m-\CQ_m^u)I_H u|_1.
	\end{equation}
\end{theorem}
Note that the assumption $A_L(x, \nabla u, \nabla u)=A(x, \nabla u)$ is satisfied for the Newton-type and the Ka\v{c}anov-type linearization introduced in Example~\ref{ex:linearization}.

\begin{proof}[Proof of Theorem \ref{thm:errorgal}]
	Since \eqref{eq:lodgal} defines a Galerkin method, the quasi-optimality \eqref{eq:quasiopt} leads to
	\[|u-u_{H,m}|_1\lesssim \inf_{v_{H,m}\in V_{H,m}}|u-v_{H,m}|_1.\]
	
	\emph{Proof of \eqref{eq:errorgalenergy}:}
	We choose $v_{H,m}=(\operatorname{id}-\CQ_m)I_H u = (\operatorname{id}-\CQ)I_H u+(\CQ-\CQ_m)I_H u$ and observe that
	the second term can directly be estimated using Proposition \ref{prop:errortrunc}, the stability of $I_H$ and \eqref{eq:stabu}.
	Note that by definition $u-(\operatorname{id}-\CQ)I_H u\in W$.
	Hence we obtain with the strong monotonicity of $A$ (cf. Assumption \ref{asspt:monotone}), the approximation property \eqref{eq:IHapprox}, and the definition of $\mathcal{Q}$ in \eqref{eq:correctionlinearized} that
	\begin{align*}
		&\!\!\!\!|u-(\operatorname{id}-\CQ)I_H u|^2_1\\*
		&\lesssim \int_{\Omega}[A(x,\nabla u)-A(x, \nabla(\operatorname{id}-\CQ)I_Hu)]\cdot \nabla(u-(\operatorname{id}-\CQ)I_H u)\, dx\\
		&=(f, u-(\operatorname{id}-\CQ)I_H u)_\Omega-\int_{\Omega}A(x,  \nabla(\operatorname{id}-\CQ)I_Hu)\cdot \nabla(u-(\operatorname{id}-\CQ)I_H u)\, dx\\
		&\stackrel{\eqref{eq:correctionlinearized}}{=}(f, u-(\operatorname{id}-\CQ)I_H u)_\Omega\\
		&\qquad-\int_{\Omega}[A(x,  \nabla(\operatorname{id}-\CQ)I_Hu)-A_L(x, \nabla u^*, \nabla \operatorname{id}-\CQ)I_H u)]\cdot \nabla(u-(\operatorname{id}-\CQ)I_H u)\, dx\\
		&\lesssim (H\|f\|_0+\eta_{\mathrm{lin}}((\operatorname{id}-\CQ)I_H u))\,|u-(\operatorname{id}-\CQ)I_H u|_1.
	\end{align*}
	Combination of this estimate with Proposition \ref{prop:errortrunc} leads to \eqref{eq:errorgalenergy}.
	
	\emph{Proof of \eqref{eq:errorgalenergy1}:}
	We again choose $v_{H,m}=(\operatorname{id}-\CQ_m)I_H u = (\operatorname{id}-\CQ)I_H u+(\CQ-\CQ_m)I_H u$, but we treat the first term in a slightly different manner.
	Namely, employing Assumption \ref{asspt:linearization} and the definition of $\CQ$ in \eqref{eq:correctionlinearized}, we deduce
	\begin{align*}
		|u-(\operatorname{id}-\CQ)I_H u|_1^2&\lesssim \int_\Omega \VA_L(x, \nabla u^*)\nabla (u-(\operatorname{id}-\CQ)I_H u)\cdot \nabla (u-(\operatorname{id}-\CQ) I_H u)\, dx\\
		&=\int_\Omega \bigl[A_L(x, \nabla u^*, \nabla u)-A_L(x, \nabla u^*, \nabla (\operatorname{id}-\CQ)I_H u)\bigr]\cdot  \nabla (u-(\operatorname{id}-\CQ)I_H u)\, dx\\
		&=\int_\Omega \bigl[A_L(x, \nabla u^*, \nabla u)-A(x,\nabla u)+A(x,\nabla u)\bigr]\cdot \nabla(u-(\operatorname{id}-\CQ)I_Hu)\, dx\\
		&=(f, u-(\operatorname{id}-\CQ)I_H u)_\Omega\\
		&\qquad +\int_\Omega \bigl[A_L(x,\nabla u^*, \nabla u)-A(x, \nabla u)\bigr]\cdot \nabla(u-(\operatorname{id}-\CQ)I_Hu)\, dx\\
		&\lesssim (H\|f\|_0+\eta_{\operatorname{lin}}(u)) |u-(\operatorname{id}-\CQ)I_H u|_1.
	\end{align*}
	Combination of this estimate with Proposition \ref{prop:errortrunc} leads to \eqref{eq:errorgalenergy1}.
	
	\emph{Proof of \eqref{eq:errorgalenergy2}:}
	Again, we choose $v_{H,m}:=(\operatorname{id}-\CQ_m)I_H u$, but we split it in a different way this time.
	We write 
	\[(\operatorname{id}-\CQ_m)I_H u=(\operatorname{id}-\CQ^u)I_H u+(\CQ^u-\CQ_m ^u)I_H u+(\CQ_m^u-\CQ_m)I_Hu,\]
	where we recall that $\CQ^u$ and $\CQ_m^u$ are the solutions to \eqref{eq:correctionlinearized} and \eqref{eq:lincorrectorlocal}, respectively, with coefficient $\mathfrak{A}^u:=\VA_L(x, \nabla u)$. 
	The last term $(\CQ_m^u-\CQ_m)I_Hu$ is directly included in \eqref{eq:errorgalenergy2}, while the second term $(\CQ^u-\CQ_m ^u)I_H u$ is again estimated with Proposition \ref{prop:errortrunc}.
	For the first term $(\operatorname{id}-\CQ^u)I_H u$, we observe due to Assumption \ref{asspt:linearization} that
	\begin{align*}
		|u-(\operatorname{id}-\CQ^u)I_H u|_1^2&\lesssim \int_\Omega \VA_L(x, \nabla u)\nabla (u-(\operatorname{id}-\CQ^u)I_H u)\cdot \nabla (u-(\operatorname{id}-\CQ^u) I_H u)\, dx\\
		&=\int_\Omega \bigl[A_L(x, \nabla u, \nabla u)-A_L(x, \nabla u, \nabla (\operatorname{id}\!-\!\CQ^u)I_H u)\bigr]\!\cdot\!  \nabla (u-(\operatorname{id}\!-\!\CQ^u)I_H u)\, dx.
	\end{align*}
	Since $u-(\operatorname{id}-\CQ^u)I_H u\in W$, the definition of $\CQ^u$ in \eqref{eq:correctionlinearized} implies that 
	\[\int_\Omega A_L(x, \nabla u, \nabla (\operatorname{id}-\CQ^u)I_H u)\cdot  \nabla (u-(\operatorname{id}-\CQ^u)I_H u)\, dx=0.\]
	Hence, employing the assumption $A_L(x, \nabla u, \nabla u)=A(x, \nabla u)$ and \eqref{eq:IHapprox}, we deduce
	\begin{equation*}
		\begin{split}
			|u-(\operatorname{id}-\CQ^u)I_H u|_1^2&\lesssim \int_\Omega A_L(x, \nabla u, \nabla u)\cdot \nabla(u-(\operatorname{id}-\CQ^u)I_Hu)\, dx\\
			&=\int_\Omega A(x, \nabla u)\cdot \nabla(u-(\operatorname{id}-\CQ^u)I_H u)\, dx\\
			&=(f, u-(\operatorname{id}-\CQ^u)I_H u)_\Omega\\
			&\lesssim H\|f\|_0\,  |u-(\operatorname{id}-\CQ^u)I_H u|_1.\hfill\qedhere
		\end{split}
	\end{equation*}
\end{proof}

Up to the linearization errors, all variants of the previous theorem are identical to the linear elliptic case \cite{MP14LOD}.
In particular, if we choose $m\approx |\log(H)|$, we have linear convergence without any assumptions on the regularity of $u$ or the variations of $A$.
By Friedrich's inequality, the same estimate also holds for the $L^2$-norm.
In contrast to the linear case, the Aubin-Nitsche trick cannot be applied so that higher order convergence for nonlinear problems is rather difficult to achieve, see  the discussion in \cite{AH16nonlinelliptic}.
Using the idea of the elliptic projection in \cite{AH16nonlinelliptic}, we obtain an $L^2$-estimate in Theorem \ref{thm:errorgalL2} in the Appendix.
Roughly speaking, it yields quadratic convergence (up to (new) linearization errors) for the choice $m\approx |\log (H)|$.

By the stability of $I_H$ we deduce an estimate for the error to $I_H u_{H,m}$, which describes the finite element part of the Galerkin solution.
\begin{corollary}\label{cor:errorgalFEM}
	Let Assumptions \ref{asspt:monotone} and \ref{asspt:linearization} be fulfilled.
	Let $u$ be the solution to \eqref{eq:nonlinearweak} and $u_{H,m}$ the solution to  \eqref{eq:lodgal}.
	Then it holds that
	\begin{align}\label{eq:errorgalFEM}
		\|u-I_Hu_{H,m}\|_0\lesssim H \inf_{v_H\in V_H}|u-v_H|_1+\|u-u_{H,m}\|_0+H|u-u_{H,m}|_1.
	\end{align}
\end{corollary}
\begin{proof}
	With the triangle inequality we split
	\[\|u-I_Hu_{H,m}\|_0\leq \|u-I_H u\|_0+\|I_H(u-u_{H,m})\|_0,\]
	which finishes the proof together with the properties (stability, approximation, and projection) of $I_H$. 
\end{proof}

Note that the two last terms in \eqref{eq:errorgalFEM} can be estimated via Theorem \ref{thm:errorgal}. 
For the $L^2$-norm we also have the estimates from Theorem \ref{thm:errorgalL2} in the appendix.
For $m\approx|\log (H)|$, the error $\|u-I_Hu_{H,m}\|_0$ is at least of order $H$ and might be up to order $H^2$ if $u$ is sufficiently regular.
If we assume $L^2$-stability of $I_H$, i.e., $\|I_H v\|_0\lesssim \|v\|_0$ for all $v\in H^1_0(\Omega)$, the estimate in Corollary \ref{cor:errorgalFEM} simplifies to
\[\|u-I_H u_{H,m}\|_0\lesssim \inf_{v_H \in V_H}\|u-v_H\|_0+\|u-u_{H,m}\|_0.\]
Together with Theorem \ref{thm:errorgalL2} this implies that the error in the finite element part of $u_{H,m}$ is dominated by the $L^2$-best-approximation error in the finite element space.

\subsection{Linearization errors and choice of linearization points}
\label{subsec:errorlinear}
In this section, we discuss estimates for the linearization errors of Theorem \ref{thm:errorgal}, namely $\eta_{\mathrm{lin}}(v)$ and $|(\CQ_m-\CQ_m^u)I_H u|_1$, and their implication on the choice of the linearization point $u^*$.
\paragraph{Linearization error $\eta_{\mathrm{lin}}(v)$.}
First, we note that due to Assumptions \ref{asspt:monotone} and \ref{asspt:linearization}, $\eta_{\mathrm{lin}}(v)$ can always be bounded as follows
\begin{align*}
	\eta_{\mathrm{lin}}(v)&\leq \|A(x,\nabla v)-A_L(x, \nabla u^*, \nabla v)\|_0\\
	&\leq \|A(x, \nabla v)-A(x,0)\|_0+\|\VA_L(x, \nabla u^*)\nabla v\|_0+\|A(x,0)\|_0+\|b_L(x, \nabla u^*)\|_0\\
	&\lesssim (\Lambda+\Lambda_L)| v|_1+C_0+\|b_L(x, \nabla u^*)\|_0.
\end{align*}
Hence, the stabilities of $\CQ$ and $I_H$ as well as \eqref{eq:stabu} directly imply that
\[\eta_{\mathrm{lin}}(u)\leq C\|f\|_0+C(u^*)\quad \text{and} \quad\eta_{\mathrm{lin}}((\operatorname{id}-\CQ)I_Hu)\leq C\|f\|_0+C(u^*)\]
with $C(u^*)=C_0+\|b_L(x,\nabla u^*)\|_0$. Considering Example~\ref{ex:linearization}, we have $C(u^*)=C_0$ for Ka\v{c}anov-type linearizations and $C(u^*)\leq 2C_0+(\Lambda+\Lambda_L)|u^*|_1$ for Newton-type lineariaztions so that in both cases the linearization error is bounded by the data $f$ and $u^*$.
As a consequence, if the nonlinearity and $\|f\|_0$ are small, the linearization error is negligible. 
This of course is a rather restrictive assumption since it basically means that we are still in the almost linear case with only a small nonlinear perturbation.
Note that the above bound on $\eta_{\mathrm{lin}}(v)$ can be simplified in the case $v=(\operatorname{id}-\CQ)I_H u$ because due to \eqref{eq:correctionlinearized}
\[\int_\Omega A_L(x,\nabla u^*, (\operatorname{id}-\CQ)I_H u)\cdot \nabla w\, dx =0 \qquad \text{for all}\quad w\in W.\]

Next, we will show that $\eta_{\mathrm{lin}}(v)$ is small for the linearizations of Example \ref{ex:linearization} if $u^*$ is close to the exact solution $u$.

\begin{lemma}[Linearization error for Ka\v{c}anov-type linearization]
	\label{lem:linKacanov}
	Let $A(x, \nabla v)=\alpha(x, |\nabla v|^2)\nabla v$ and set $\VA_L(x, \nabla u^*)=\alpha(x,|\nabla u^*|^2)$ and $b_L(x, \nabla u^*)=0$. Then,
	\[\eta_{\mathrm{lin}}(v)\leq (\Lambda+\Lambda_L) |v-u^*|_1.\]
\end{lemma}
\begin{proof}
	Assumptions \ref{asspt:monotone} and \ref{asspt:linearization} directly yield
	\begin{align*}
		\eta_{\operatorname{lin}}(v)&\leq \|\alpha(x,|\nabla v|^2)\nabla v - \alpha(x,|\nabla u^*|^2)\nabla v\|_0\\
		&\leq \|A(x,\nabla v)-A(x,\nabla u^*)\|_0+\|\alpha(x,|\nabla u^*|^2)\nabla (u^*-v)\|_0\\
		&\leq (\Lambda+\Lambda_L) \|\nabla (v-u^*)\|_0.\qedhere
	\end{align*}
\end{proof}

\begin{lemma}[Linearization error for Newton-type linearization]
	\label{lem:linNewton}
	Set $\VA_L(x,\nabla u^*)=D_\xi A(x, \nabla u^*)$ and \linebreak[4]$b_L(x, \nabla u^*)=A(x, \nabla u^*)-D_\xi A(x, \nabla u^*)\nabla u^*$.
	Assume that $D_\xi A$ is Lipschitz continuous in its last argument, i.e., there is $L_A>0$ such that
	\begin{equation}\label{eq:assptLipschitz}
		|D_\xi A(x, \xi_1)-D_\xi A(x, \xi_2)|\leq L_A |\xi_1-\xi_2|\qquad\text{for all}\quad \xi_1, \xi_2\in \rz^d\quad \text{and almost all}\quad x\in \Omega.
	\end{equation}
	Then,
	\[\eta_{\mathrm{lin}}(v)\leq L_A \|\nabla (v-u^*)\|_{L^\infty(\Omega)}\, |v-u^*|_1.\]
\end{lemma}

\begin{proof}
	We perform a Taylor expansion of $A$ around $u^*$ and obtain that for any $w\in H^1_0(\Omega)$ it holds that
	\begin{align*}
		&\!\!\!\!\int_\Omega \bigl[ A(x,\nabla v)-A(x, \nabla u^*)+D_\xi A(x, \nabla u^*)\nabla(v-u^*)\bigr]\cdot \nabla w\, dx\\
		&=\int_{\Omega}\int_0^1\bigl[ D_\xi A(x, \nabla u^*+\tau \nabla(v-u^*))-D_\xi A(x, \nabla u^*)\bigr]\nabla (v-u^*)\cdot \nabla w\, d\tau \, dx.
	\end{align*}
	The assumed Lipschitz continuity of $D_\xi A$ \eqref{eq:assptLipschitz} finishes the proof.
\end{proof}

Note that we did not exploit $w\in W$ in the estimates for $\eta_{\mathrm{lin}}(v)$.
Lemmas \ref{lem:linKacanov} and \ref{lem:linNewton} underline that $u^*$ should be close to $u$, which can be seen as a local convergence result and motivates the discussion of different linearization points further below in this section.

\paragraph{Linearization error $|(\CQ_m-\CQ_m^u)I_H u|_1$.}
In case of a Ka\v{c}anoc-type linearization from Example \ref{ex:linearization}, we observe that $\eta_{\mathrm{lin}}(u)$ in \eqref{eq:errorgalenergy1} can be bounded as $\eta_{\mathrm{lin}}(u)\leq \|\mathfrak{A}^u-\mathfrak{A}\|_{L^\infty (\Omega)}|u|_1$.
Hence, this linearization error can be expressed as the $L^\infty(\Omega)$-error between the ``true'' coefficient $\mathfrak{A}^u$ and the ``reference'' coefficient $\mathfrak{A}$ used in practice for computing the corrector $\CQ_m$.
The error $|(\CQ_m-\CQ_m^u)I_H u|_1$ in \eqref{eq:errorgalenergy2} is obviously also closely related to the error between the coefficients $\mathfrak{A}^u$ and $\mathfrak{A}$.
However, estimating directly the error between the correction operators allows for a refined estimate.
Before we present this in detail, let us first note that, obviously,
\[|(\CQ_m-\CQ_m^u)I_H u|_1\leq \sup_{v_H\in V_H, |v_H|_1=1}|(\CQ_m-\CQ_m^u)v_H|_1\, |I_H u|_1\]
so that we will estimate $(\CQ_m-\CQ_m^u)v_H$ for an arbitrary $v_H\in V_H$ in the following.
Further, by the definition of $\CQ_m$ and $\CQ_m^u$ as sums of element correctors with support only in $\UN^m(T)$ we have
\[|(\CQ_m-\CQ_m^u)v_H|_1^2\lesssim C_{\mathrm{ol},m}\sum_{T\in \CT_H}|(\CQ_{T, m}-\CQ_{T,m}^u)v_H|_1^2.\]
We now have the following estimate for the error between the correction operators, the proof is postponed to the Appendix \ref{app:proofcorrecerror}.

\begin{proposition}\label{prop:correcerror}
	Fix $T\in \CT_H$ and $v_H\in V_H$ and recall the notation $\UN^m(T)$ from Section~\ref{subsec:local}.
	For any $a\in L^\infty(\UN^m(T), \rz^{d\times d})$ denote by $\overline{a}$ the average of the trace, i.e.,
	$\overline{a}:=\fint_{\UN^m(T)} \mathrm{tr}(a)\, dx$, and denote by $\widehat{a}\in L^\infty(\UN^m(T),\rz^{d\times d})$ the scaled coefficient, i.e., $\widehat{a}:=a \big/\overline{a}$.
	Recall $\mathfrak{A}=\boldsymbol{A}_L(\cdot, \nabla u^*)$ with associated corrector $\CQ_{T,m}$ and $\mathfrak{A}^u=\boldsymbol{A}_L(\cdot, \nabla u)$ with associated corrector $\CQ_{T,m}^u$, cf.~Assumption~\ref{asspt:linearization}.
	Define
	\[E_{\CQ, T}^2:=\sum_{T^\prime\in \CT_H, T^\prime\subset \UN^m(T)}\|\widehat{\mathfrak{A}^u}-\widehat{\mathfrak{A}}\|_{L^\infty(T^\prime)}^2\max_{\psi|_T, \psi\in V_H}\frac{\|(\chi_T\nabla \psi-\nabla \CQ_{T,m}\psi)\|_{0,T^\prime}^2}{| \psi|_{1,T}^2}.\]
	Then,
	\[|(\CQ_{T,m}-\CQ_{T,m}^u)v_H|_1\lesssim E_{\CQ, T} |v_H|_{1,T}.\]
\end{proposition}
We emphasize that Proposition \ref{prop:correcerror} relates the error between the element correctors to the error between the coefficients $\mathfrak{A}$ and $\mathfrak{A}^u$, but (i) locally on each element or element patch, respectively, and (ii) only the $L^\infty$-error between the \emph{scaled} coefficients is relevant.
The latter point is possible because \eqref{eq:lincorrectorlocal} can be multiplied by the \emph{scalar-valued} constant $\overline{\mathfrak{A}}$ without changing the element corrector.
Hence, if a coefficient is (locally) multiplied by a constant, the error between the associated correction operators is zero while the error between the (unscaled) coefficients themselves can be very large. 

Error indicators similar to $E_{\CQ, T}$ in Proposition \ref{prop:correcerror} were already presented in \cite{HM19lodsimilar, HKM20lodperturbed} in the context of time-changing or perturbed diffusion coefficients with the following differences.
First, the idea of studying the scaled coefficients is new in Proposition \ref{prop:correcerror}. Second, since $\mathfrak{A}^u$ is not available in practice in our case, Proposition \ref{prop:correcerror} is an a priori result and aims at linking the error between the correction operators to the error between the coefficients.
This is in sharp contrast to the previous works, which use $E_{\CQ,T}$ as a practical indicator to steer re-computation of element correctors. 
For this reason, some terms in \cite{HKM20lodperturbed,HM19lodsimilar} are (slightly) different to obtain contrast-independent bounds. 
The result of Proposition \ref{prop:correcerror} can probably be refined in this spirit as well, but we omit this for simplicity.

\paragraph{Choice of linearization points $u^*$.}
The previous discussion has shown that the choice of $u^*$ is crucial for the performance of the method and that, ideally, $u^*$ should be chosen close the exact solution $u$.
With this in mind, let us compare some possible choices of $u^*$ also used in the numerical experiments.

\begin{enumerate}
	\item We can select for $u^*$ the initial value of the nonlinear iteration, for instance $u^*=0$. This results in a multiscale space $V_{H,m}$ that can be computed a priori and yields a very cheap method.
	The overall error of course will only be dominated by the discretization error of order  $H+\beta^m$ if $u^*$ is already close to the exact solution.
	For instance, a similar assumption is also needed to guarantee convergence of Newton's iteration method for the nonlinear problem.
	In particular the choice $u^*=0$ should yield satisfactory results for problems without steep gradients. 
	
	\item Another option is to compute the finite element solution $u_H$ using $V_H$ and to use it as $u^*$.
	Since $V_H$ is associated with a coarse mesh, the computation of $u_H$ is rather cheap.
	Due to the finescale features of $A$, however, $u_H$ alone will not be a satisfactory approximation of $u$ in general, not even in the $L^2(\Omega)$-norm. 
	The finescale features of $A$ are then taken into account in the ensuing LOD solution.
	To estimate the linearization error in this case, we can combine Lemmas \ref{lem:linKacanov} or \ref{lem:linNewton}, respectively, with standard a priori estimates for the finite element solution, see Section~\ref{subsec:fem} for estimates in the $L^2(\Omega)$-norm and $H^1(\Omega)$-semi norm.
	In case of the Newton-type linearization, Lemma \ref{lem:linNewton} also requires a $W^{1,\infty}(\Omega)$-estimate, where we obtain from \cite[Chapter 8]{BrenScott94}
	\[\|u-u_H\|_{W^{1, \infty}(\Omega)}\leq C H \|u\|_{W^{2,\infty}(\Omega)}.\]
	Note that for all finite element error estimates, higher regularity of $u$ is required and that the constants will depend on the finescale parameter $\varepsilon$.
	Computing the LOD solution $u_{H,m}$ with $u^*=u_H$ is notably different from the two-grid approach \cite{Xu96twogridnonlinear} where the FE solution $u_H$ is ``post-processed'' by another (linear) FE solve on a finer mesh.
	
	\item Finally, one can compute a whole cascade of LOD solutions by starting with some $u^*$, computing the solution $u_{H,m}$ to \eqref{eq:lodgal}, and using this solution as new $u^*$.
	For this cascade, we can inductively combine \eqref{eq:errorgalenergy1} and Lemmas \ref{lem:linKacanov} or \ref{lem:linNewton}, respectively, to obtain an a priori error estimate, which consists of the discretization error of order $H+\beta^m$ and the linearization error for the initial choice of $u^*$.
	This straightforward error estimate, however, does not seem to be optimal because the initial linearization error will always remain and one does not exploit that  the LOD solution should be closer to the exact solution with each step.
	For this cascade of LOD solutions, an error indicator in the spirit of Proposition \ref{prop:correcerror} -- with $\CQ_m^u$ replaced by $\CQ_m^{u_{H,m}}$ with the current LOD solution -- can be used to locally determine which element correctors should be recomputed with the new linearization point.
	Since \eqref{eq:lodgal} is solved by an iterative method as well, one can also imagine to already update some element correctors during this nonlinear iteration.
	Such an iterative linearized approach is studied for the nonlinear Helmholtz equation in \cite{MV20lodNLH}, but the application to quasilinear problems is a future research topic of its own.
\end{enumerate}

\section{Numerical experiments}
\label{sec:numexp}
We present the results of several numerical experiments, subject to different multiscale coefficients and nonlinearities. In all cases, the computational domain is $\Omega=[0,1]^2$.
Since no exact solution is known in any of the examples, we compute a reference solution $u_h\in V_h$ using a standard finite element method on a fine mesh $\CT_h$ that resolves all multiscale features.
Specifically, we fix $h=2^{-8}$ and vary $H=2^{-2}, 2^{-3},\ldots 2^{-6}$. We present results for oversampling parameters $m=1,2,3$ and already point out that $m=2,3$ is a sufficient choice in most experiments.
To solve the nonlinear problems, we use Newton's method with  tolerance $10^{-11}$ for the residual as stopping criterion.
Sections \ref{subsec:expperiodic} and \ref{subsec:exprandom} consider nonlinearities of the type $A(x, \nabla u)$, where Assumption \ref{asspt:monotone} is only satisfied in Section \ref{subsec:expperiodic}.
Further, we consider a model for the stationary Richards equation with a quasilinear coefficient of the form $a(x,u)\nabla u$ in Section \ref{subsec:exprichards}.
Note that the resulting nonlinear form $\CB$ is no longer monotone such that the above proof techniques do not directly transfer, see~\cite{AV12quasilin,AV14fehmmquasilinelliptic}. Nevertheless, we use the presented multiscale method with the obvious modifications for \eqref{eq:lodgal} and $\mathfrak{A}:=a(x,0)$ in the corrector problems.
The code is available at Zenodo with doi~10.5281/zenodo.4311614 and is based on a preliminary implementation of the LOD for linear problems developed at the Chair of Computational Mathematics, University of Augsburg cf.~\cite{MP20lodbook}.

\subsection{Periodic coefficient}
\label{subsec:expperiodic}
\begin{figure}
	\includegraphics[width=0.3\textwidth, trim=20mm 75mm 22mm 75mm, clip=true, keepaspectratio=false]{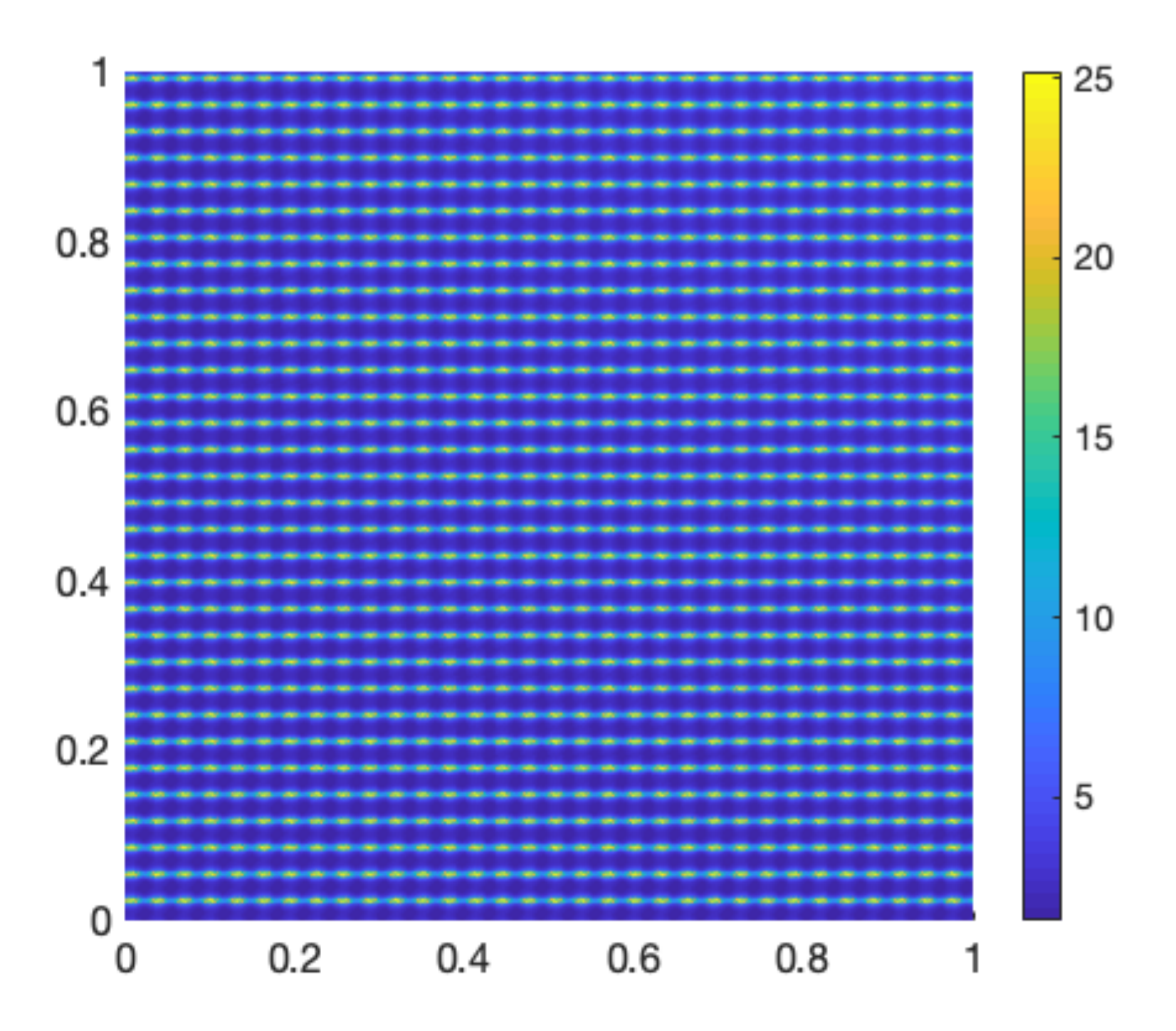}%
	\hspace{1ex}%
	\includegraphics[width=0.3\textwidth, trim=20mm 75mm 22mm 75mm, clip=true, keepaspectratio=false]{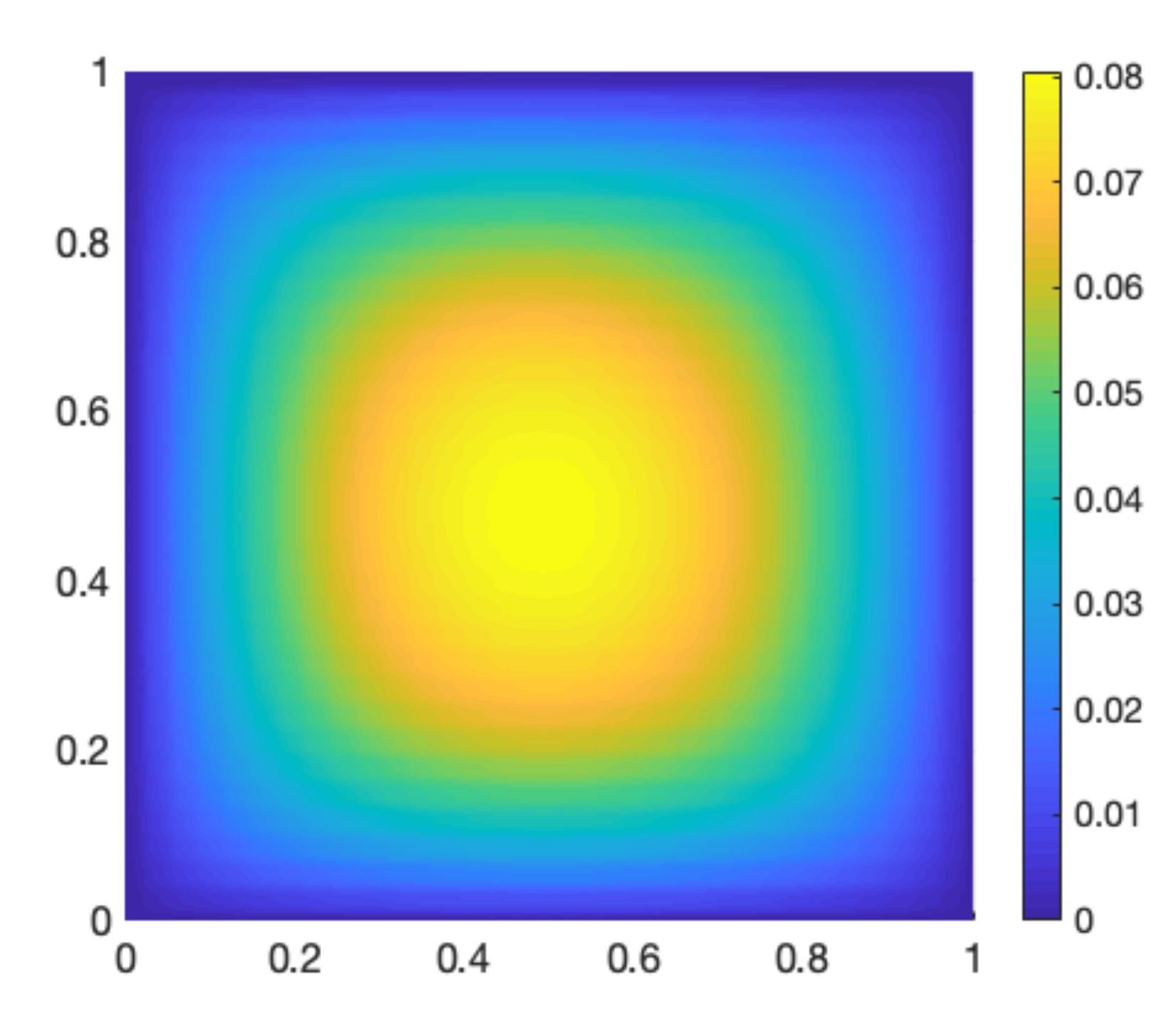}%
	\hspace{1ex}%
	\includegraphics[width=0.3\textwidth, trim=20mm 75mm 22mm 75mm, clip=true, keepaspectratio=false]{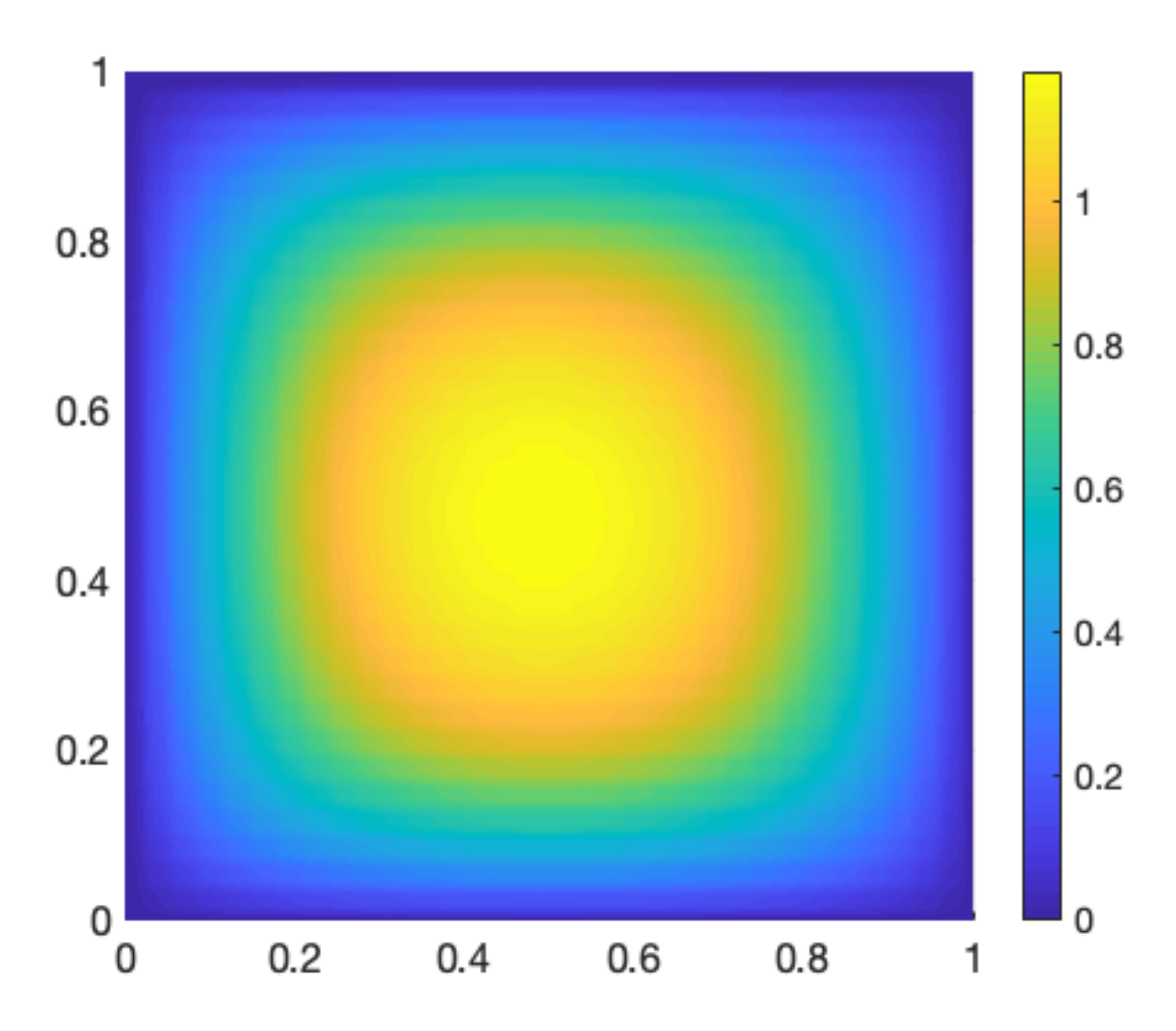}
	\caption{Spatial part of the coefficient (left) and reference solutions $u_h$  for $f_1$(middle) and $f_2$ (right) for Section~\ref{subsec:expperiodic}.}
	\label{fig:coeffsol_periodic}
\end{figure}
We choose a model problem similar to \cite{AH16nonlinelliptic} with nonlinearity
\[A(x,\xi) = \Bigl(1+x_1x_2+\frac{1.1+\frac{\pi}{3}+\sin(2\pi\frac{x}{\varepsilon})}{1.1+\sin(2\pi\frac{x}{\varepsilon})}\Bigr)\Bigl(1+\frac{1}{\bigl(1+|\xi|^2\bigr)^{1/2}}\Bigr)\xi\]
with $\varepsilon=2^{-5}$ and sources $f_1(x)=10\exp(-0.1|x-x_0|^2)$ or $f_2(x)=100\exp(-0.1|x-x_0|^2)$ with $x_0=(0.45, 0.5)^T$.
The coefficient and the reference solutions $u_h$ are depicted in Figure \ref{fig:coeffsol_periodic}. 
We will consider the two right-hand sides $f_1$ and $f_2$ to study the influence of higher values of the solution $u$ and its gradient on the errors.
Note that one can hope for higher regularity of the exact solution $u$ in this experiment.

\begin{figure}%
	\centering
	\includegraphics[width=0.47\textwidth, trim=20mm 75mm 22mm 75mm, clip=true, keepaspectratio=false]{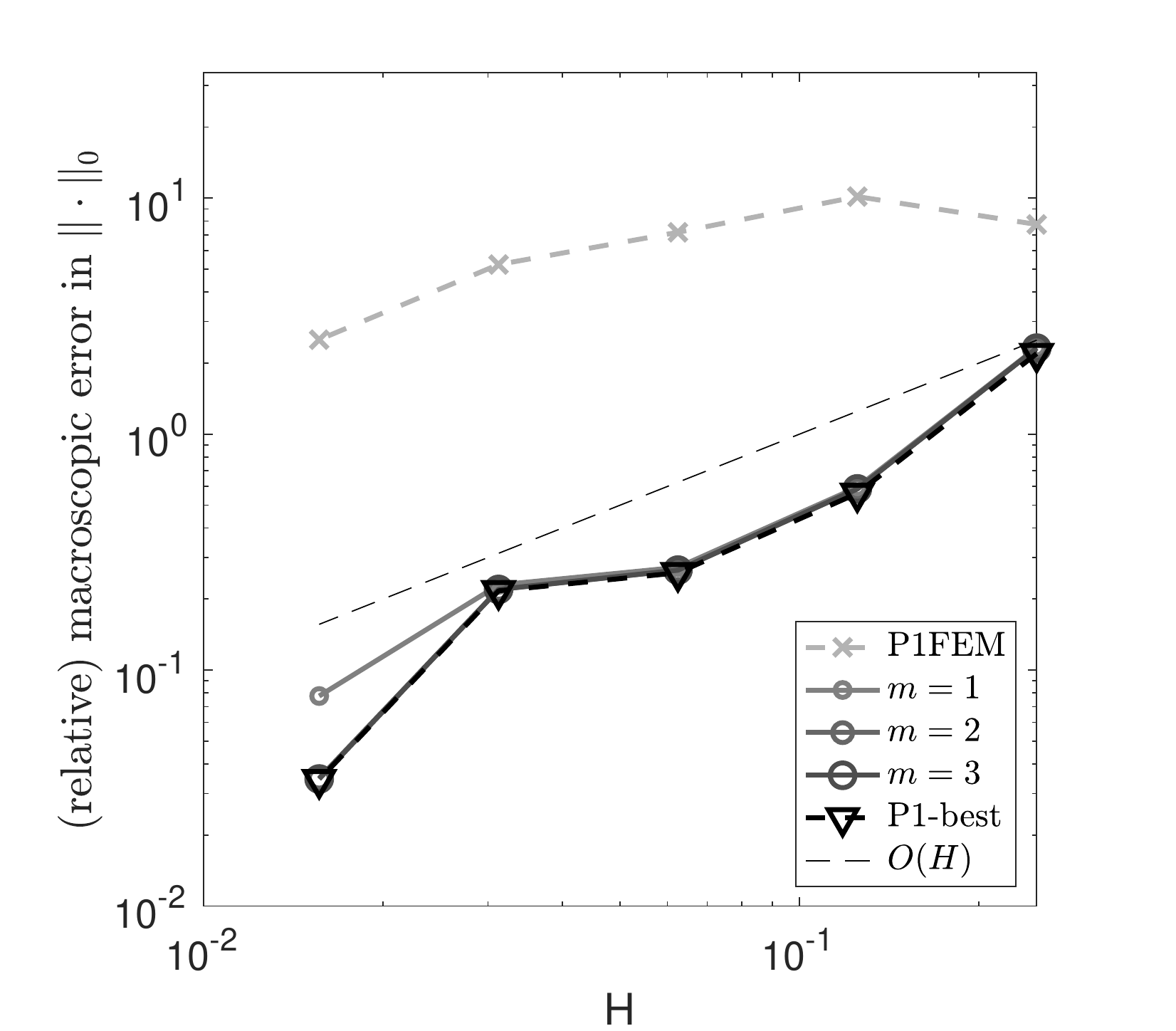}%
	\hspace{2ex}%
	\includegraphics[width=0.47\textwidth, trim=20mm 75mm 22mm 75mm, clip=true, keepaspectratio=false]{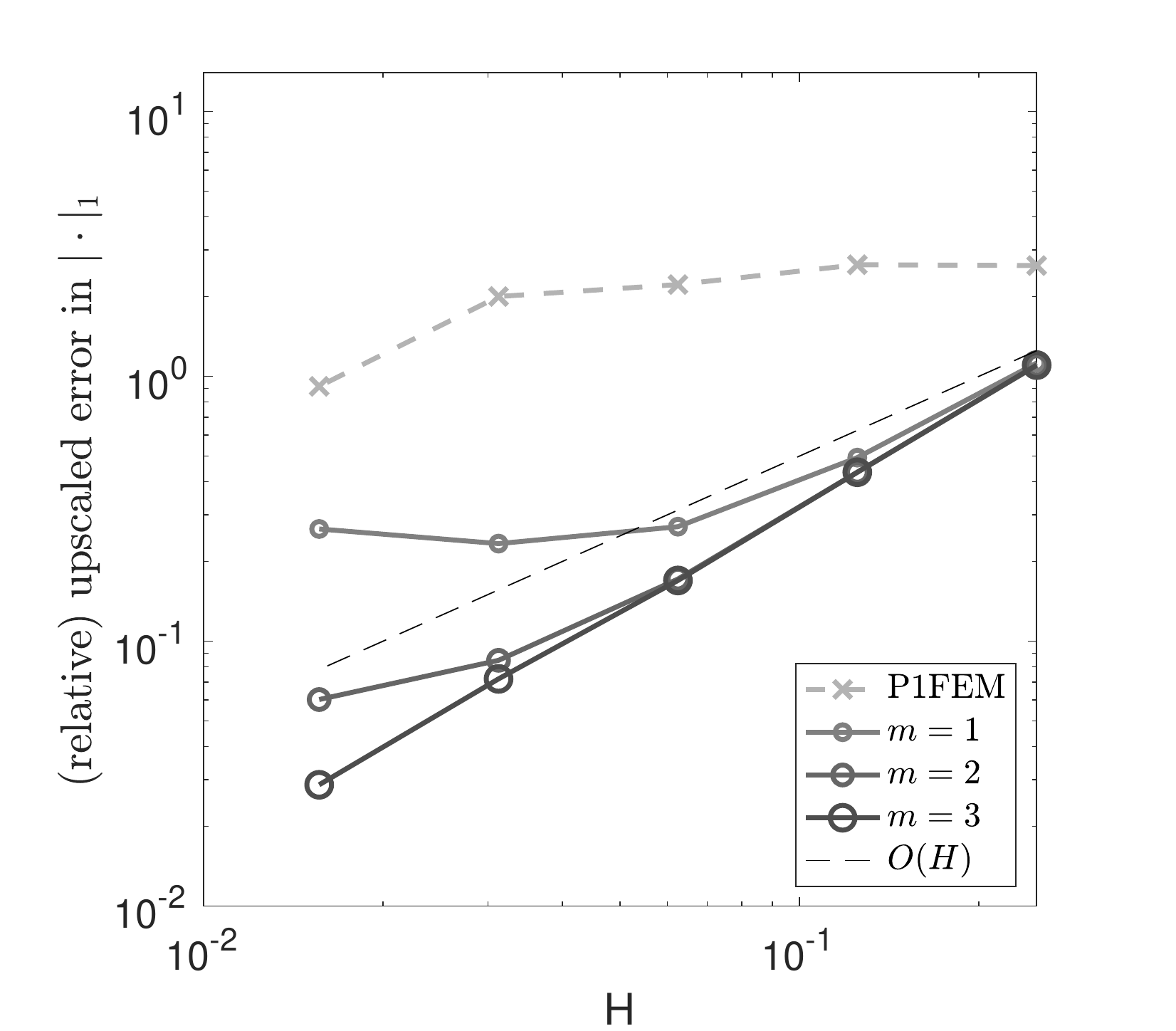}\\
	\includegraphics[width=0.47\textwidth, trim=20mm 75mm 22mm 75mm, clip=true, keepaspectratio=false]{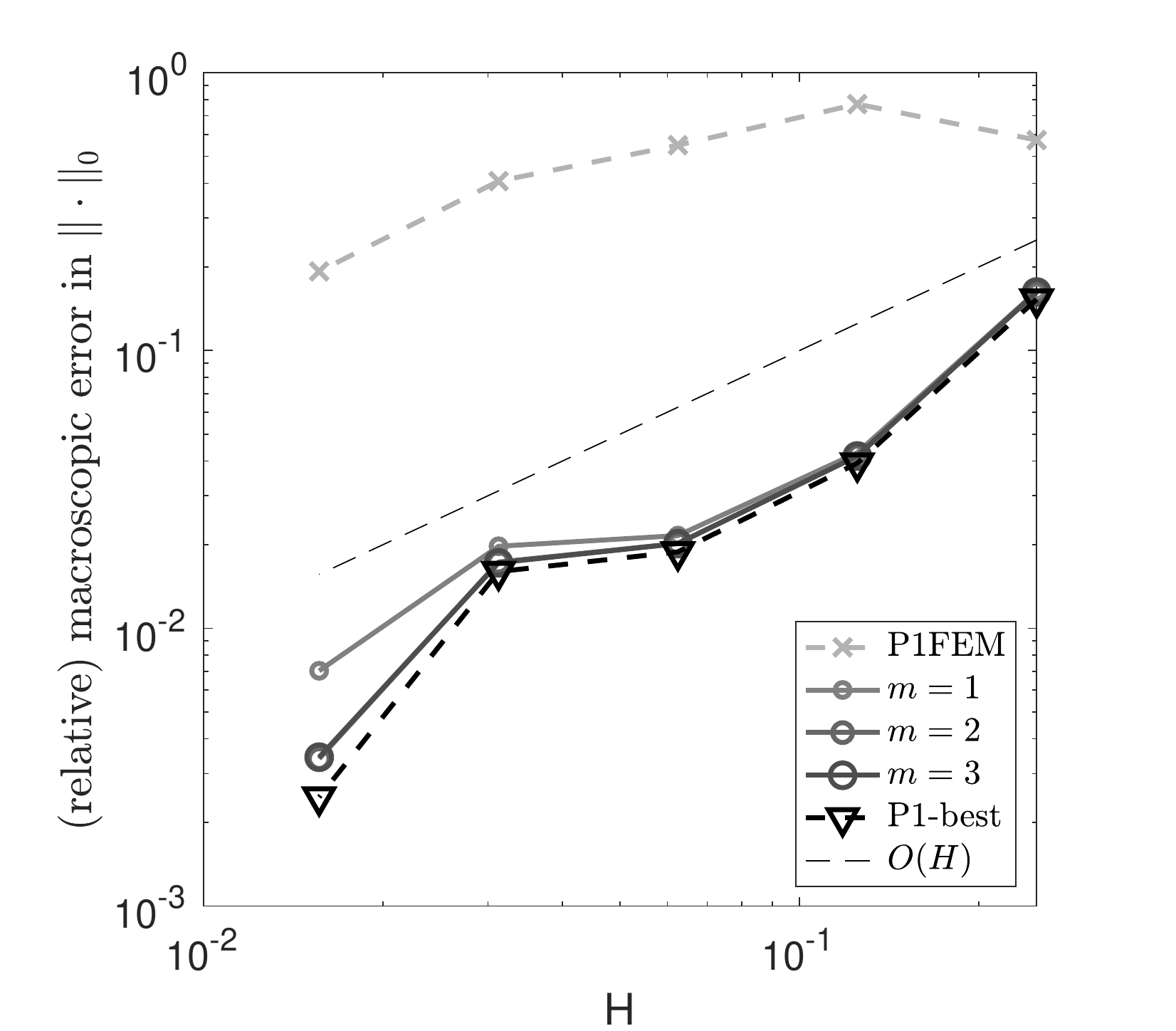}%
	\hspace{2ex}%
	\includegraphics[width=0.47\textwidth, trim=20mm 75mm 22mm 75mm, clip=true, keepaspectratio=false]{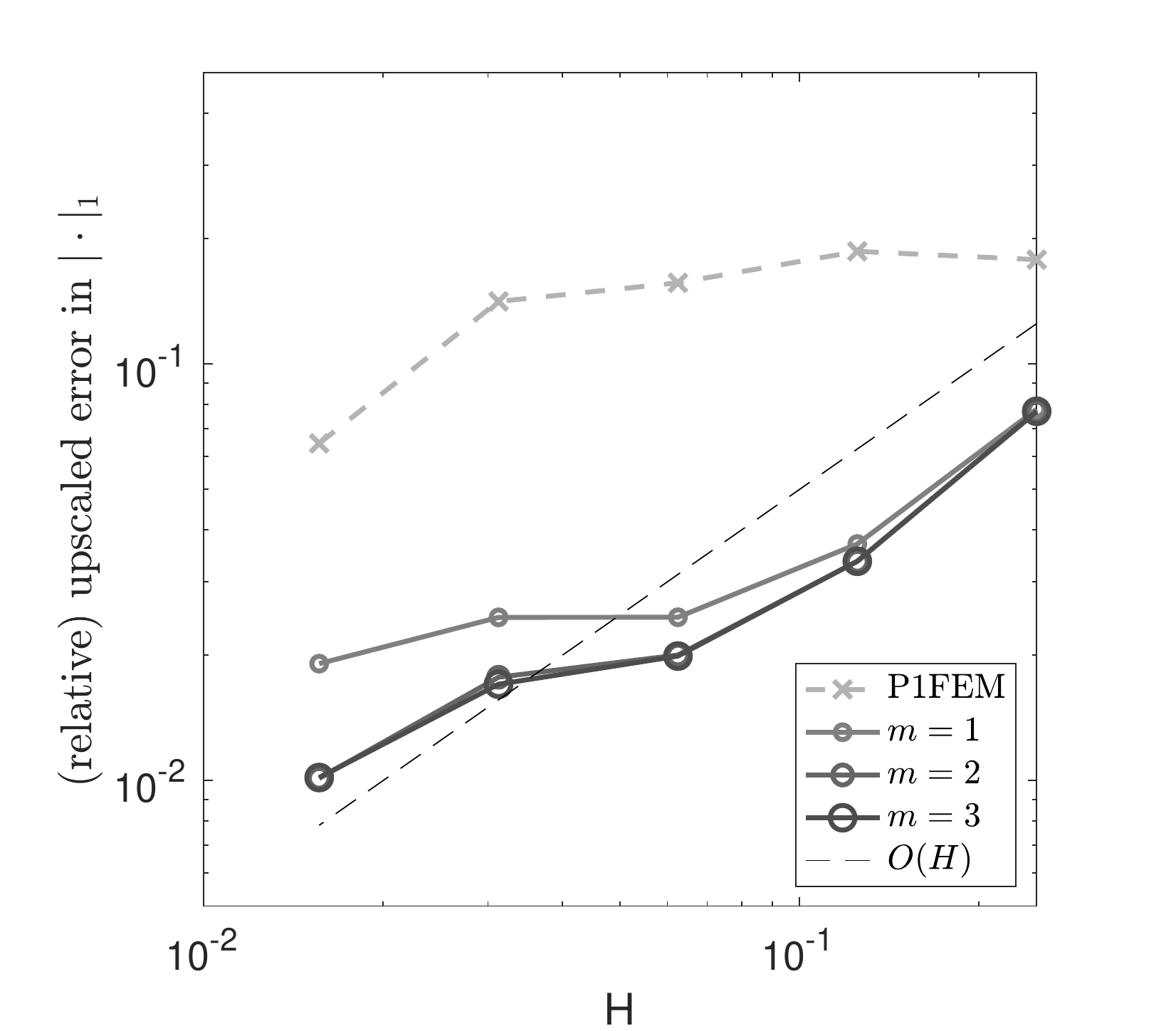}%
	\caption{Convergence histories for the (relative) macroscopic error $e_H$ (in $L^2$-norm)  (left) and the (relative) upscaled error $e_{\operatorname{LOD}}$ (in $H^1$-semi norm) (right) with $f_1$ (top) and $f_2$ (bottom) for the experiment of Section \ref{subsec:expperiodic}.}
	\label{fig:convhist_periodic}
\end{figure}
We use the Galerkin method \eqref{eq:lodgal} and obtain a solution $u_{H, m}\in V_{H, m}$, where the correctors are computed using the Newton-type linearization at $u^*$, which is specified below.
We focus on two (relative) errors in the following: The so-called \emph{(relative) upscaled error} 
\[ e_{\operatorname{LOD}}:=\frac{|u_h-u_{H,m}|_1}{|u_h|_1},\]
for which we expect a linear convergence rate (cf. Theorem \ref{thm:errorgal});
and the so-called \emph{(relative) macroscopic error}
\[e_H:=\frac{\|u_h-I_H u_{H,m}\|_0}{\|u_h\|_0},\]
for which we expect the same behavior as the $L^2$-best approximation in $V_H$ (cf. Corollary~\ref{cor:errorgalFEM}).

For the simple choice $u^*=0$, these two errors are depicted for the two right-hand sides in Figure~\ref{fig:convhist_periodic}.
We note that the (relative) macroscopic errors $e_H$ in the left column closely follow the error of the (relative) $L^2$-best approximation in the space $V_H$  for $m=2,3$ (cf.\ the discussion after Corollary \ref{cor:errorgalFEM}). Since $I_H u_{H, m}$ lies in the same space, we cannot hope for anything better. 
In particular, the reduced convergence rate for $H$ between $\sqrt{\varepsilon}$ and $\varepsilon$ is no defect of the method, but intrinsic to the problem, see also the discussion of this so-called resonance effect in \cite{GP17lodhom}.
We emphasize that in the pre-asymptotic range $\varepsilon\ll H$, the standard finite element method shows no convergence rates in contrast to the multiscale method. 
For approximations in the $H^1$-semi norm,  the (coarse-scale) space $V_H$ is no longer sufficient.
Therefore, we consider the (relative) upscaled error $e_{\operatorname{LOD}}$ in the right column of Figure~\ref{fig:convhist_periodic}.
This error overall converges linearly as expected from Theorem~\ref{thm:errorgal}. 
All in all, the experiment clearly confirms the predicted convergence rates of Theorem \ref{thm:errorgal} and Corollary \ref{cor:errorgalFEM}.

\begin{figure}
	\includegraphics[width=0.47\textwidth, trim=20mm 75mm 22mm 75mm, clip=true, keepaspectratio=false]{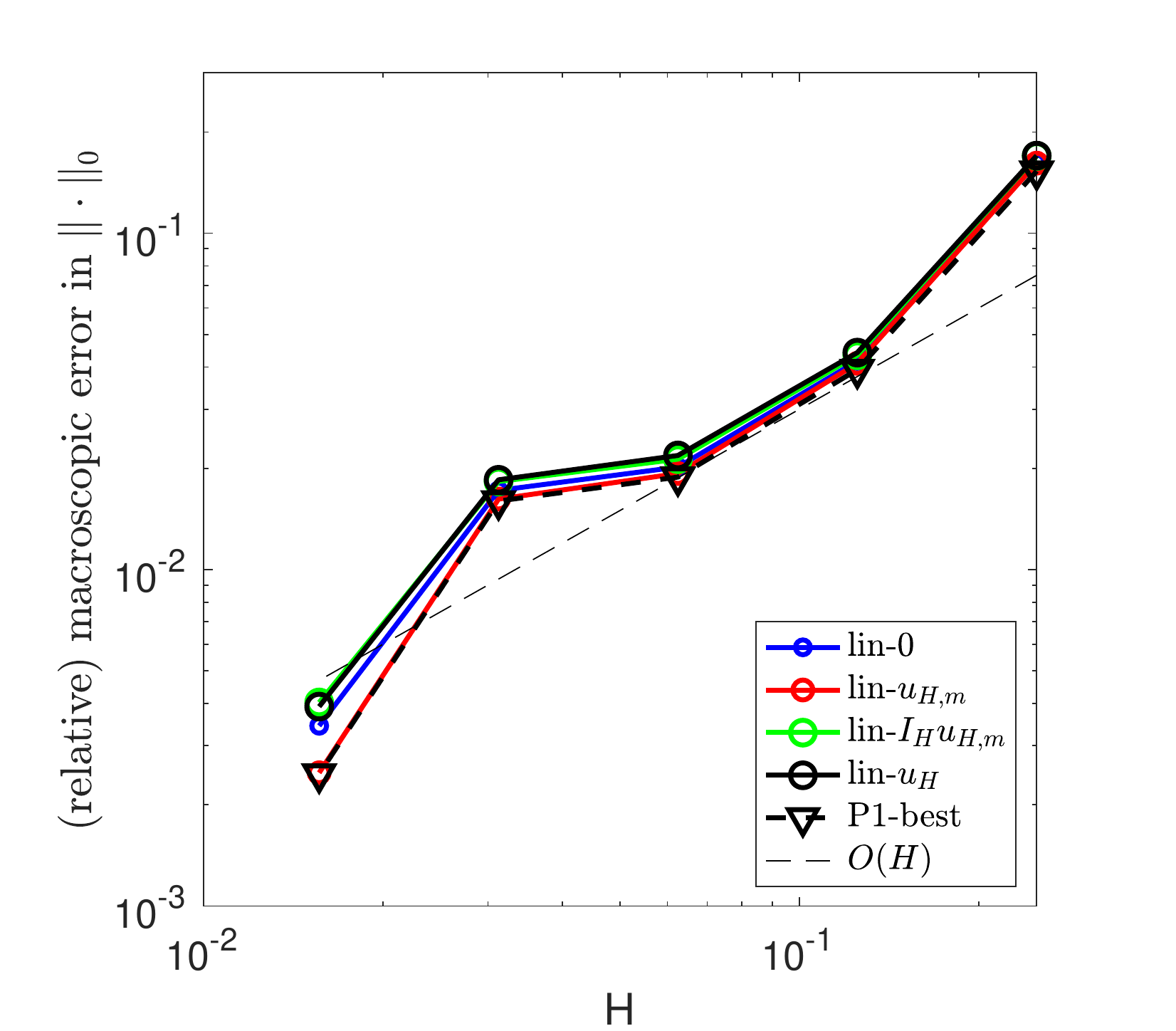}%
	\hspace{2ex}%
	\includegraphics[width=0.47\textwidth, trim=20mm 75mm 22mm 75mm, clip=true, keepaspectratio=false]{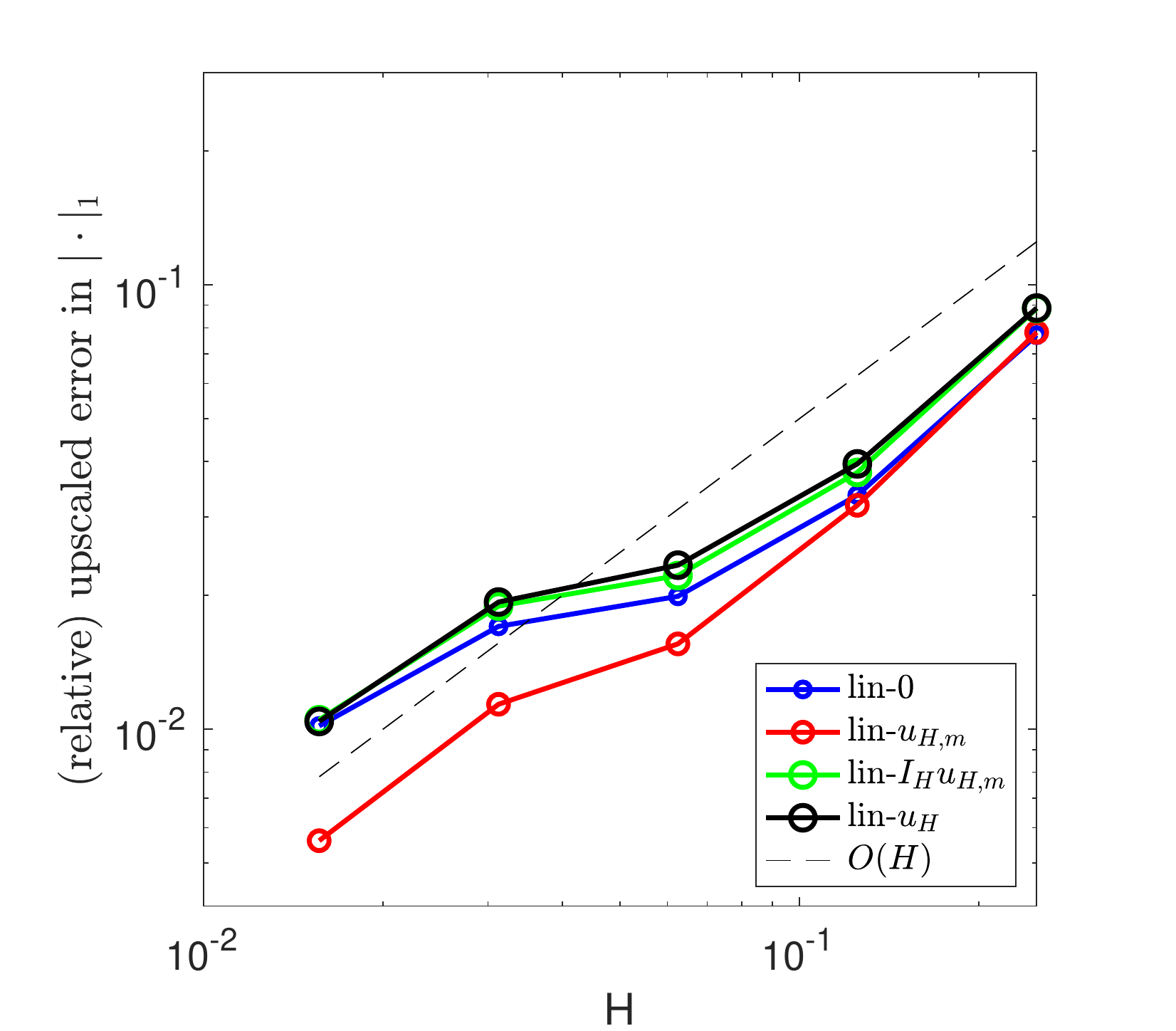}
	\caption{Convergence histories for the (relative) macroscopic error $e_H$ (in $L^2$-norm)  (left) and the (relative) upscaled error $e_{\operatorname{LOD}}$ (in $H^1$-semi norm) (right) with $f_2$ and different choices of the linearization point for the experiment of Section \ref{subsec:expperiodic}.}
	\label{fig:linconvhist_periodic}
\end{figure}
When comparing the top and bottom row of Figure \ref{fig:convhist_periodic}, we also observe that the larger gradients of $u_h$ caused by $f_2$ influence the performance of the method.
In particular $e_{\mathrm{LOD}}$ shows some deviation from the optimal linear convergence.
Therefore, we compare the behavior of $e_H$ and $e_\mathrm{LOD}$ for different linearization points with fixed $m=3$ in Figure~\ref{fig:linconvhist_periodic}.
We consider the following choices of $u^*$ as discussed in Section \ref{subsec:errorlinear}: $u^*=0$, $u^*=u_H$ with the (coarse) FE solution $u_H\in V_H$ and $u^*=u_{H,m}$ as well as $u^*=I_H u_{H,m}$, where $u_{H,m}$ is the LOD solution with linearization at zero.
All these solutions show a similar qualitative behavior, but in the quantitative errors we observe clear differences, especially for $e_\mathrm{LOD}$. 
The choices $u^*=u_H$ and $u^*=I_H u_{H,m}$ result in almost identical results because they lie in the same space $V_H$.
A bit surprisingly, those choices even perform slightly worse than $u^*=0$.  A possible explanation is that the gradients are relevant for the nonlinearity, where the space $V_H$ does not provide sufficient approximations.
Both for $e_H$ and $e_\mathrm{LOD}$, the choice $u^*=u_{H,m}$ performs best, which motivates to study iterative LOD approximations as discussed in Section \ref{subsec:errorlinear} in future research.

\subsection{Random coefficient}
\label{subsec:exprandom}
\begin{figure}
	\includegraphics[width=0.47\textwidth, trim=20mm 75mm 22mm 75mm, clip=true, keepaspectratio=false]{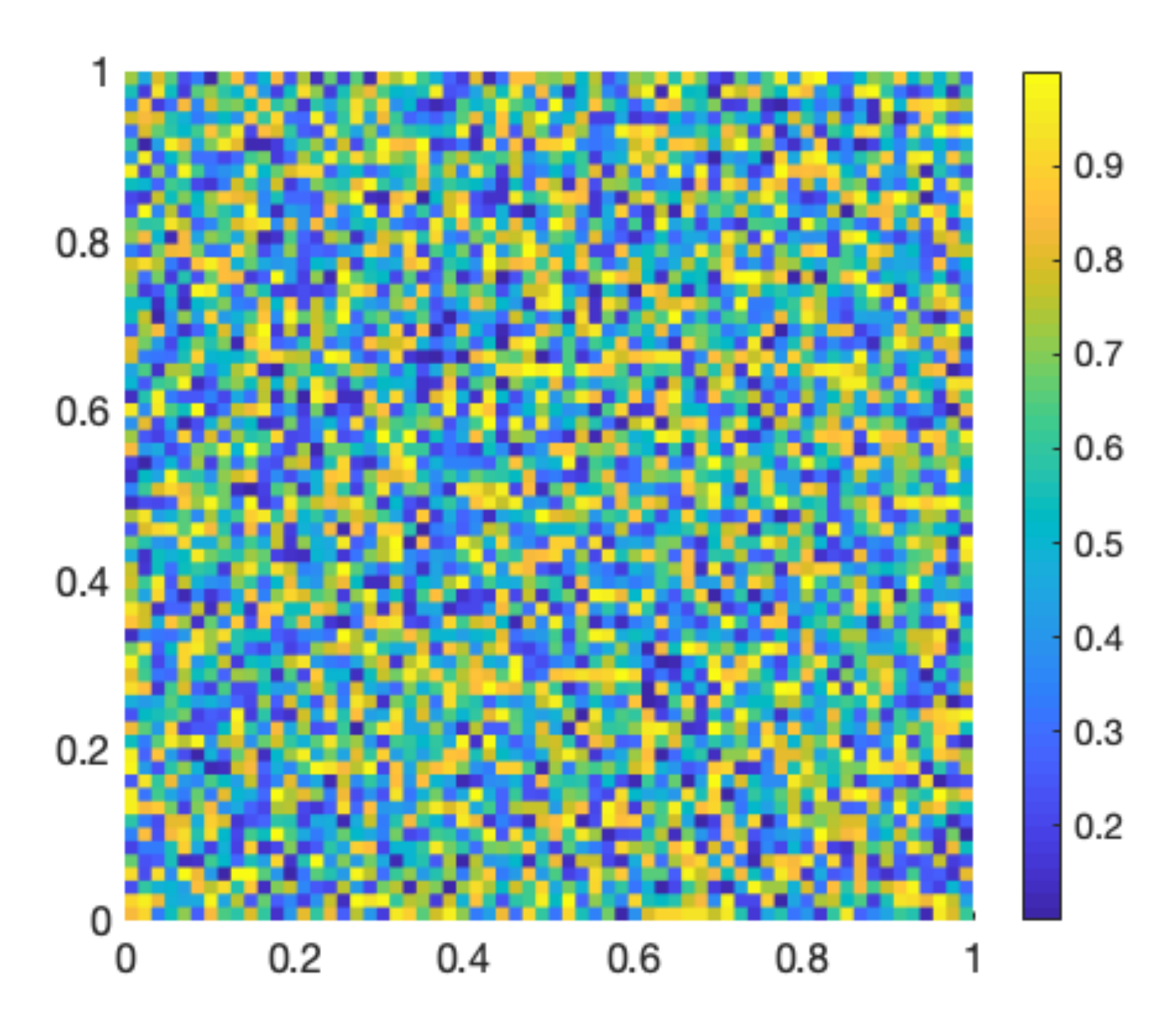}%
	\hspace{2ex}%
	\includegraphics[width=0.47\textwidth, trim=20mm 75mm 22mm 75mm, clip=true, keepaspectratio=false]{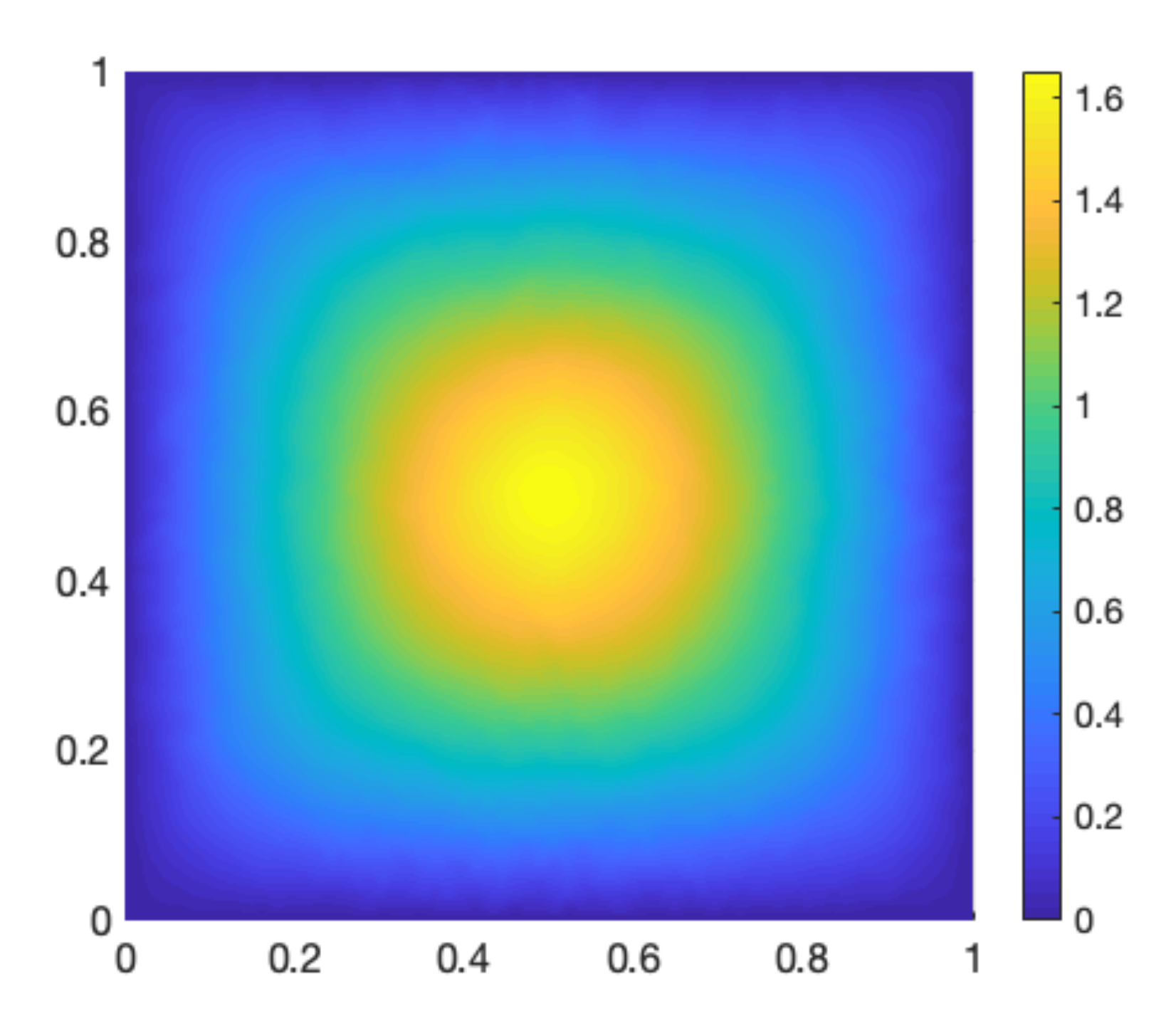}
	\caption{Spatial part of the coefficient (left) and reference solution $u_h$ (right) for the experiment of Section \ref{subsec:exprandom}.}
	\label{fig:coeffsol_random}
\end{figure}
We choose the nonlinear coefficient as
\[A(x,\xi)=c(x)\begin{pmatrix}
	\xi_1+\frac13\xi_1^3\\ \xi_2+\frac13\xi_2^3
\end{pmatrix},\]
where $c(x)$ is piece-wise constant on a quadrilateral mesh $\CT_\varepsilon$ with $\varepsilon=2^{-6}$ and the values are random numbers in $[0.1, 1]$.
The right-hand side is
\[f(x)=\begin{cases}
	5 &x_2\leq 0.1,\\50 &\text{else},
\end{cases}\]
\cite[see][]{Henn11phdhmm} for the nonlinearity and a similar right-hand side.
Note that we clearly cannot expect higher regularity for the exact solution because of the spatial discontinuities in $A$.
Further, $A$ is only locally Lipschitz constant in its second argument. Thus, Assumption \ref{asspt:monotone} is violated and we even have $|A|\to \infty$ for $|\xi|\to \infty$.
The coefficient and corresponding reference solution $u_h$ are depicted in Figure \ref{fig:coeffsol_random}.
\begin{figure}
	\includegraphics[width=0.47\textwidth, trim=20mm 75mm 22mm 75mm, clip=true, keepaspectratio=false]{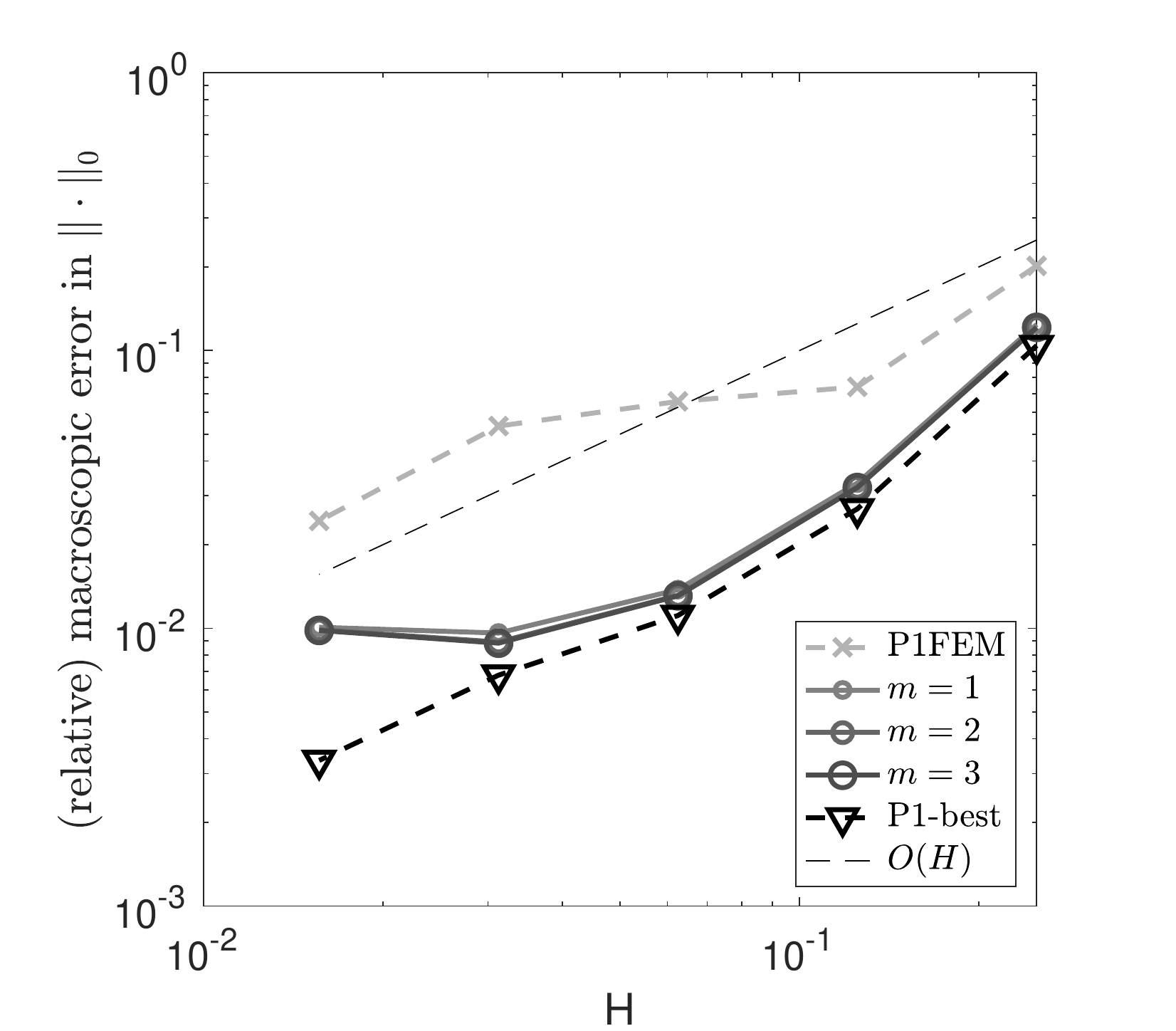}%
	\hspace{2ex}%
	\includegraphics[width=0.47\textwidth, trim=20mm 75mm 22mm 75mm, clip=true, keepaspectratio=false]{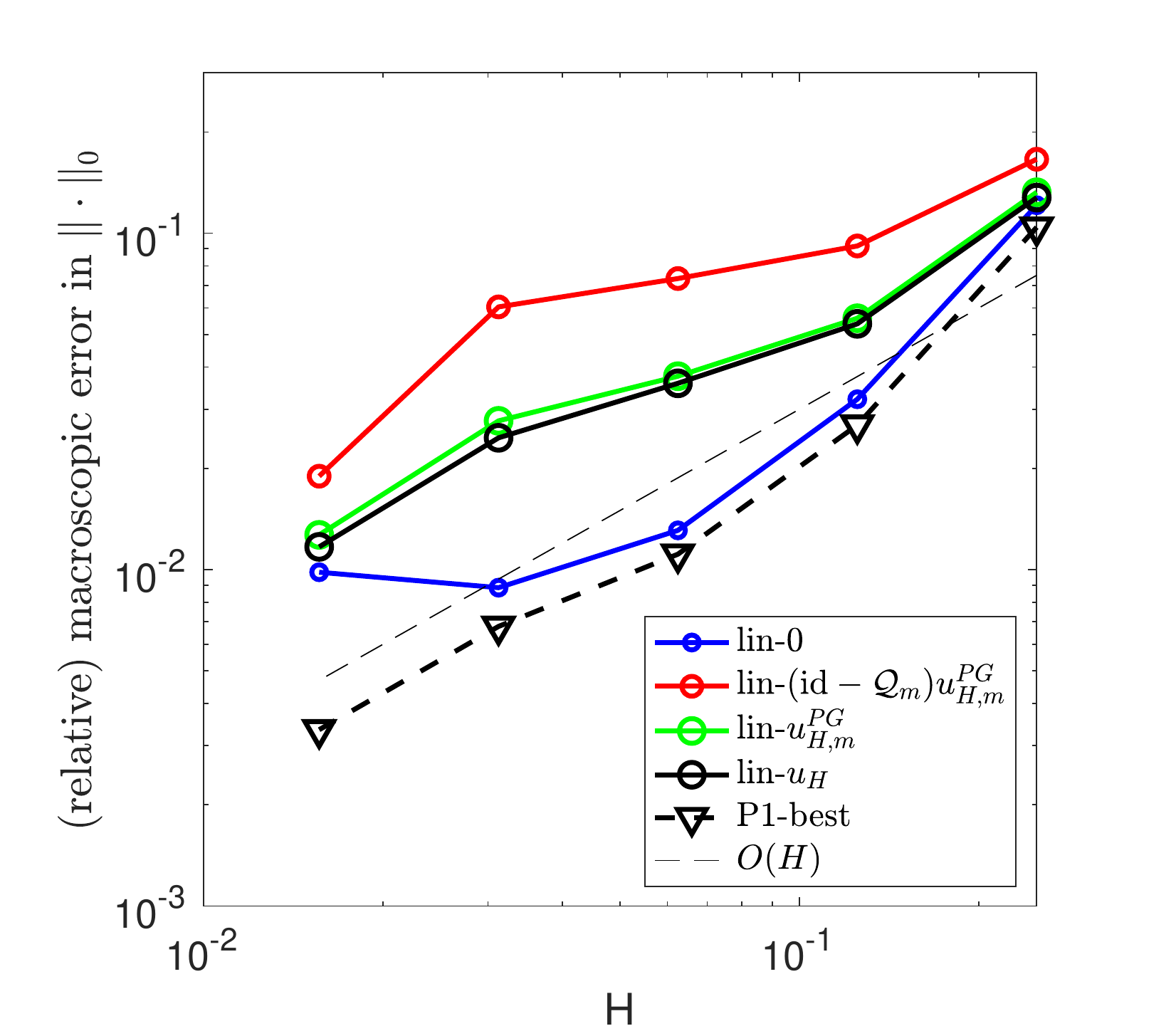}
	\caption{Convergence histories for the (relative) macroscopic error $e_H$ (in $L^2$-norm) for fixed $u^*=0$ and different $m$ (left) and for fixed $m=3$ and different $u^*$ (right) for the experiment of Section \ref{subsec:exprandom}.}
	\label{fig:convhistPGL2_random}
\end{figure}    
For this example, we study the Petrov-Galerkin LOD as briefly discussed in Remark \ref{rem:PGLOD}. Precisely, we compute $u_{H,m}^{PG}\in V_H$ as the solution of
\[\CB(u_{H,m}^{PG}; v_H)=(f, v_H)_\Omega\qquad \text{for all}\quad v_H\in V_{H,m},\]
where $V_{H,m}$ is defined as in the Galerkin case.
Hence, the solution $u_{H,m}^{PG}$ lies in the FE space and we expect the relative macroscopic error
$e_H:=\frac{\|u_h-u_{H,m}\|_0}{\|u_h\|_0}$
to follow the $L^2$-best approximation, cf.~Corollary~\ref{cor:errorgalFEM}.  

We first consider the convergence history of $e_H$ for  $u^*=0$ and different choices of $m$ in Figure \ref{fig:convhistPGL2_random} (left).
The multiscale method performs obviously better than the standard FEM. The error of the Petrov-Galerkin LOD is following the $L^2$-best approximation as expected up to a saturation or stagnation for the last considered mesh. Since this effect occurs for all $m=1,2,3$ and was also observed for $m=4,5$ (data not shown), the linearization error most probably starts to dominate.
As in the previous example, we also consider $e_H$ for fixed $m$ and different $u^*$ in Figure \ref{fig:convhistPGL2_random} (right). The considered choices are $u^*=0$, $u^*=u_H$ with the FE solution $u_H$ on the coarse mesh as before and $u^*=u_{H,m}^{PG}$ as well as $u^*=(\operatorname{id}-\CQ_m)u_{H,m}^{PG}$, where $u_{H,m}^{PG}$ is the Petrov-Galerkin LOD solution with $u^*=0$ and $\CQ_m$ is the corrector computed with $u^*=0$.
The latter two linearization points play roles comparable to $I_H u_{H,m}$ and $u_{H,m}$ in the Galerkin setting discussed in the previous experiment.
Only $u^*=0$ shows the saturation at the end, but nevertheless, this choice of linearization performs best.
We emphasize that the impractical choice of $u^*=u_h$ would lead to an $e_H$ completely following the $L^2$-best approximation error (data not shown). This illustrates the validity of our error estimates also for the Petrov-Galerkin LOD, but also stresses the influence of the chosen linearization.
Moreover, this example confirms and underlines that the present multiscale method does not rely on assumptions such as periodicity or scale separation.

\subsection{Stationary Richards equation}
\label{subsec:exprichards}
\begin{figure}
	\includegraphics[width=0.47\textwidth, trim=20mm 75mm 22mm 75mm, clip=true, keepaspectratio=false]{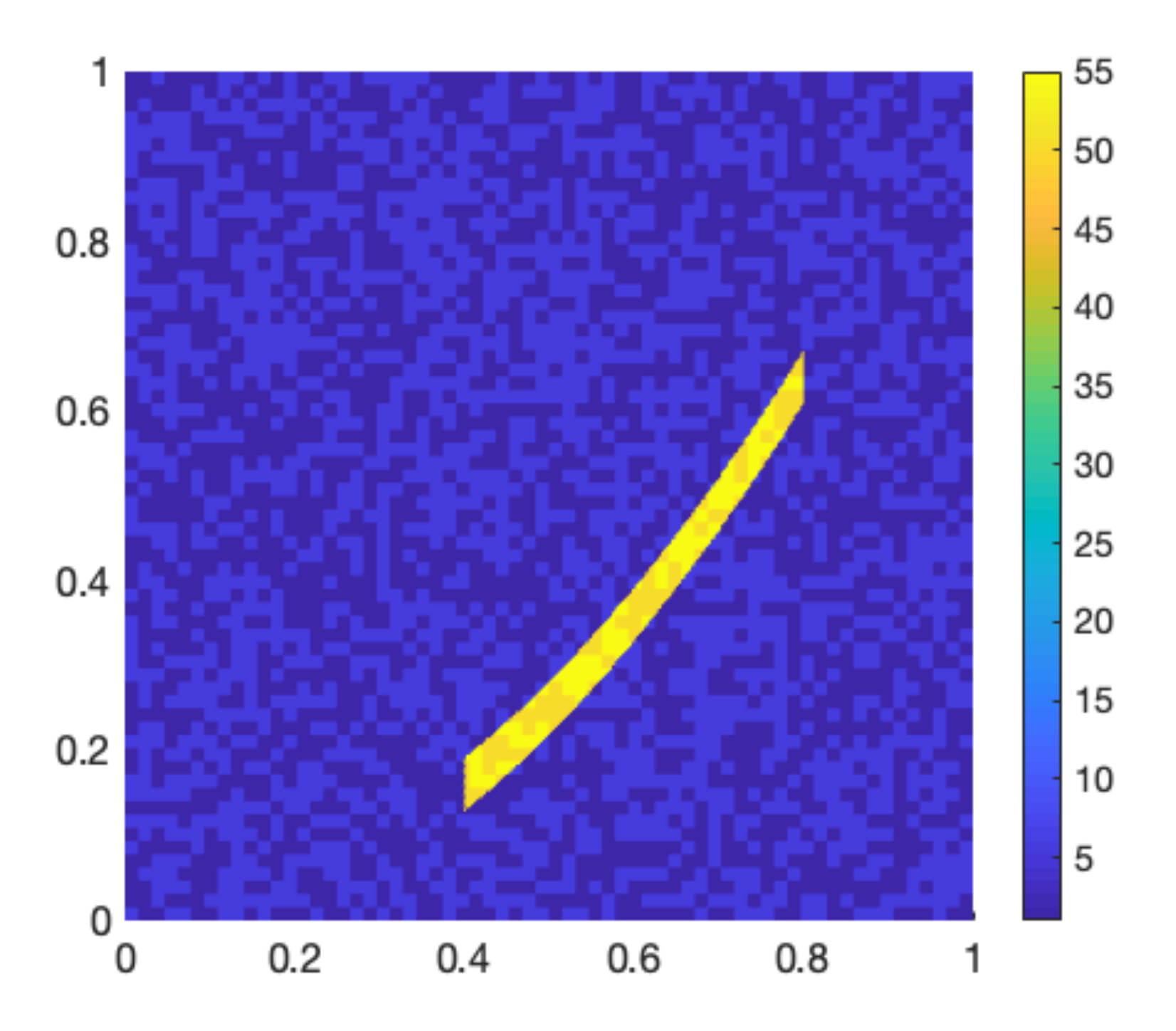}%
	\hspace{2ex}%
	\includegraphics[width=0.47\textwidth, trim=20mm 75mm 22mm 75mm, clip=true, keepaspectratio=false]{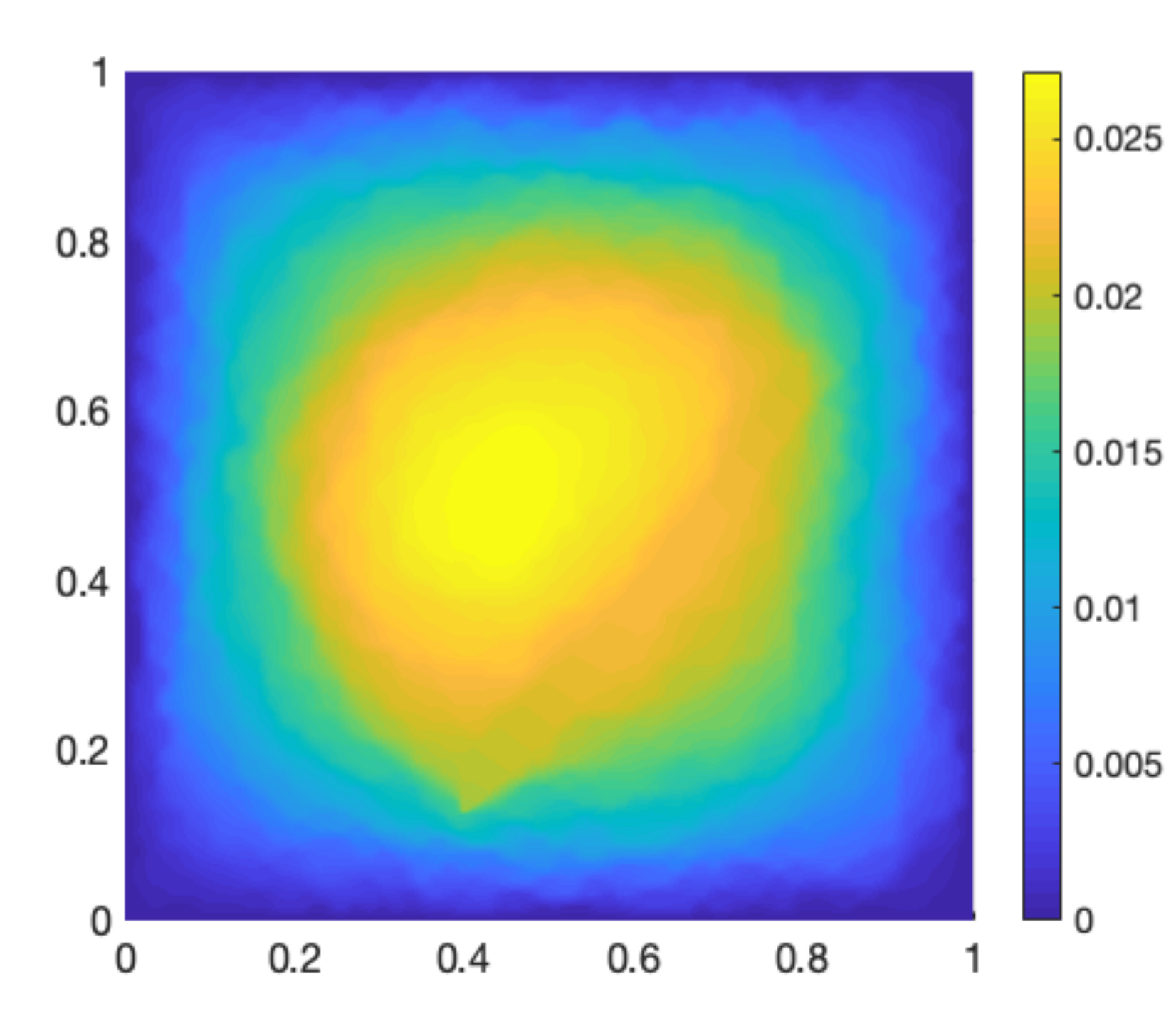}%
	\caption{Coefficient $c(x)$ (left) and reference solution (right) for the experiment of Section \ref{subsec:exprichards}.}
	\label{fig:coeffsol_quasilin}
\end{figure} 
We now test the applicability of our method to quasilinear non-monotone problems, which, for instance, are frequently encountered in (unsaturated) groundwater flow and can be modeled by the (stationary) Richards equation.
Here, we consider a quasilinear coefficient of the form $A(x,u,\nabla u)=c(x)k(u)\nabla u$.
For $c$ we choose a spatial multiscale model with a channel as depicted in Figure \ref{fig:coeffsol_quasilin} left, which is often present in geophysical applications. Note that this can easily be extended to the case of several channels.
For the nonlinearity $k$, we consider the following so-called van Genuchten model \cite{vgen80genuchtenmodel}
\[k(s)=\frac{(1-\alpha|s|(1+(\alpha|s|)^2)^{-1/2})^2}{1+(\alpha|s|)^2}
\]
with $\alpha=0.005$, \cite[see also][]{EGKL12nonlinflowitsolver}.
The right-hand side is 
\[f(x)=\begin{cases}
	0.1 &x_2\leq 0.1,\\1 &\text{else},
\end{cases}\]
The reference solution for the present setting is depicted in Figure \ref{fig:coeffsol_quasilin} right.
Note the influence of the channel, which analytically manifests itself in a low regularity of the solution.  

\begin{figure}
	\includegraphics[width=0.47\textwidth, trim=20mm 75mm 22mm 75mm, clip=true, keepaspectratio=false]{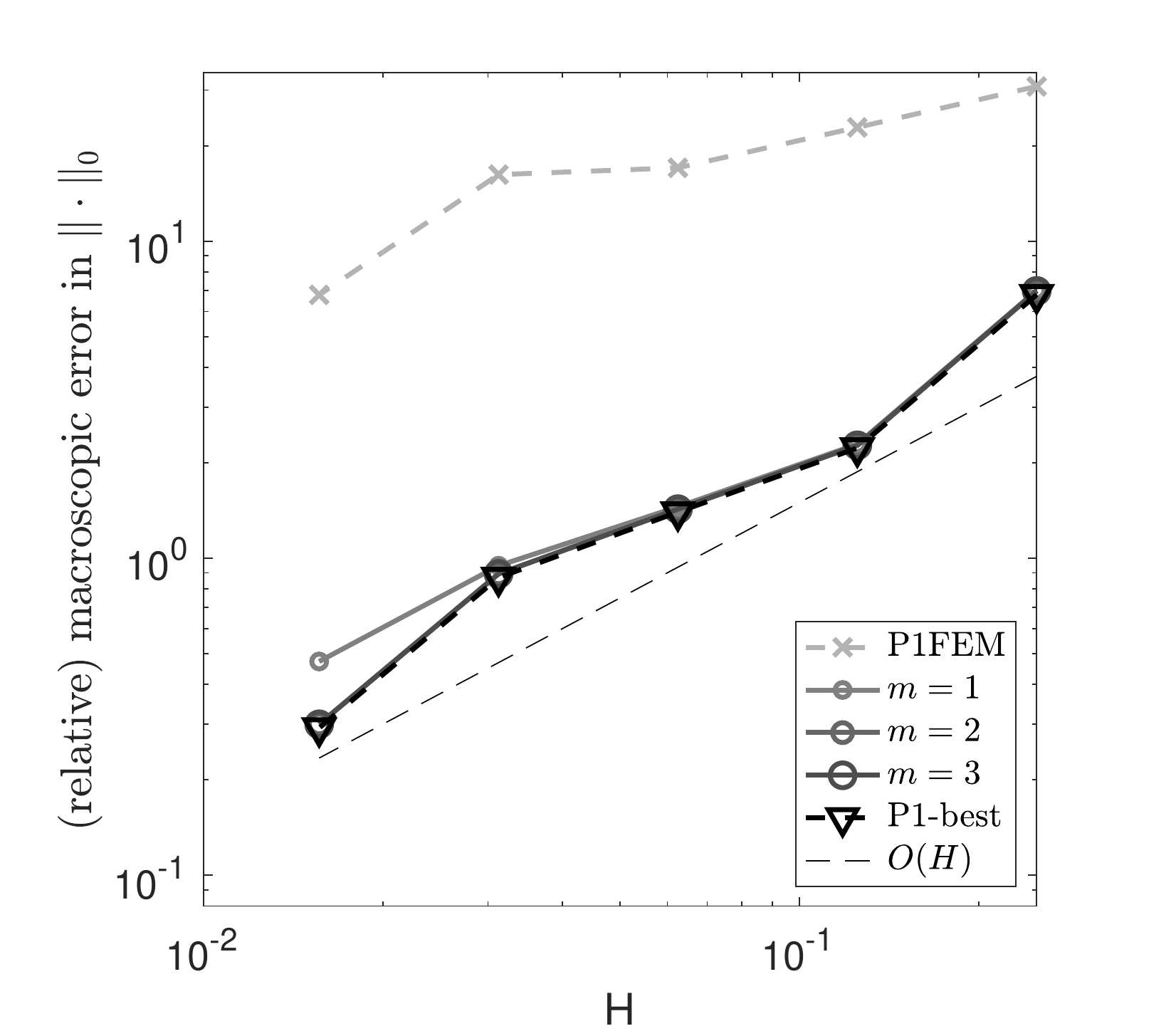}%
	\hspace{2ex}%
	\includegraphics[width=0.47\textwidth, trim=20mm 75mm 22mm 75mm, clip=true, keepaspectratio=false]{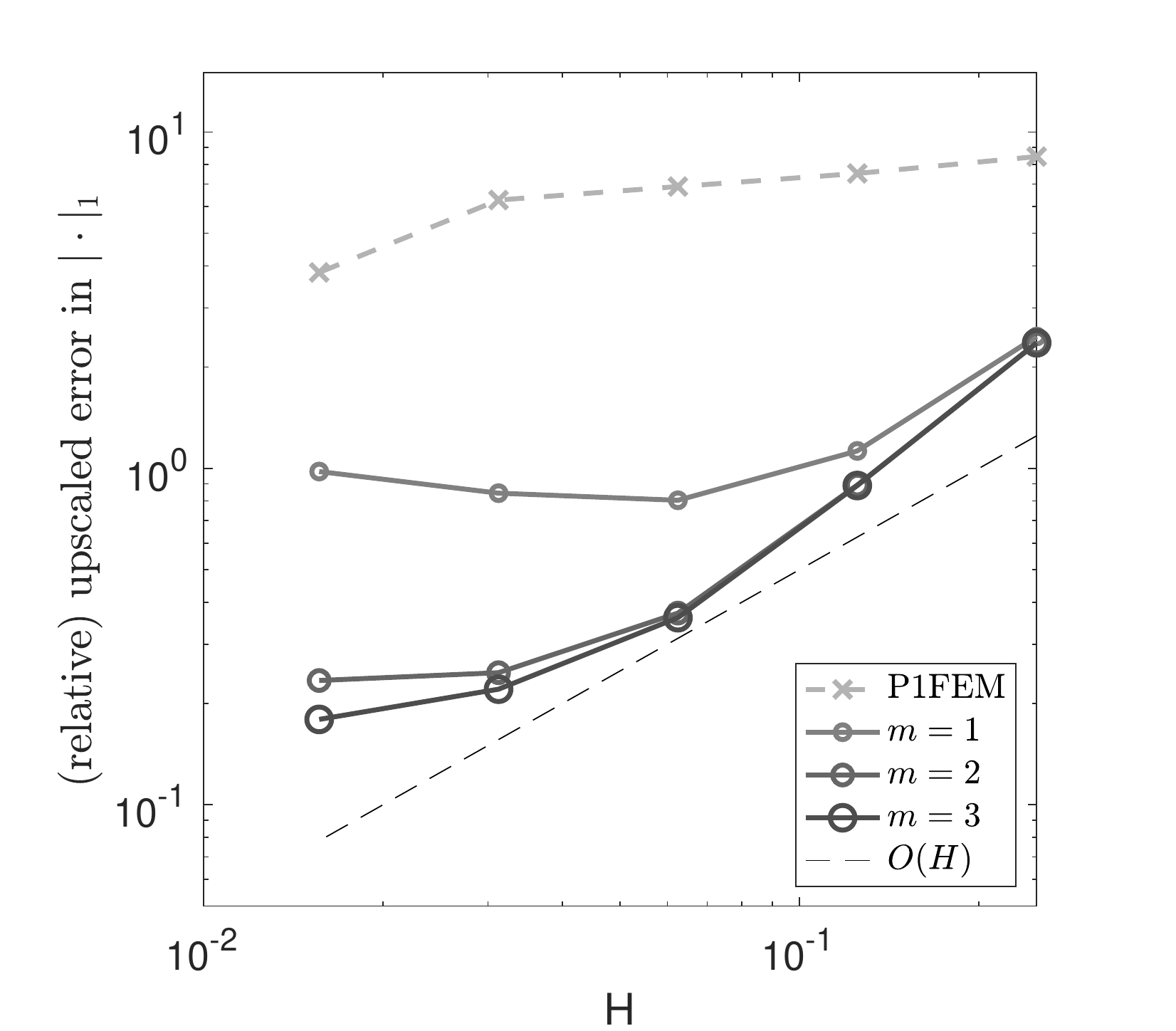}%
	\caption{Convergence histories for the (relative) macroscopic error (in $L^2$-norm) (left) and for the (relative) upscaled error $e_{\operatorname{LOD}}$ (in $H^1$-semi norm) (right) for the experiment of Section \ref{subsec:exprichards}.}
	\label{fig:convhist_quasilin}
\end{figure} 
As a consequence, we observe exactly the (worst-case) convergence rates as predicted by our theory (for the monotone case), but not more. In particular, $e_H$  follows the best-approximation error, which in this case is ``only'' linear as discussed after Corollary \ref{cor:errorgalFEM}.
Furthermore, we see the expected linear convergence of the upscaled error $e_{\operatorname{LOD}}$ up to a slight saturation for small $H$.
This is most probably caused again by a dominating linearization error, which could be cured by a different choice of the linearization.
This experiment clearly underlines the applicability of the approach also beyond the strictly monotone case and indicates that similar convergence rates as in Theorem \ref{thm:errorgal} and Corollary \ref{cor:errorgalFEM} can be expected.

\section*{Conclusion}
We presented a multiscale method for nonlinear monotone elliptic problems with spatial multiscale features.
A problem-adapted multiscale basis is constructed by solving local linear fine-scale problems for each coarse-scale mesh element.
Numerical analysis shows optimal error estimates up to linearization errors and we discussed choices of the linearization. 
Several numerical experiments underline and confirm the applicability of the method as well as the expected convergence rates. We also numerically compared the influence of the chosen linearization on the performance of the method. 
As mentioned, iterative (adaptive) LOD methods where either a cascade of LOD solutions to different linearization points is computed or the correctors are (partially) updated during the nonlinear iteration are interesting extensions of the presented method and will be subject of future research.
Also the numerical analysis for the quasilinear non-monotone problems will be studied in the future.

\section*{Acknowledgments}
We are grateful to D.~Peterseim for fruitful discussion on the subject and for providing a preliminary implementation of the LOD for linear problems. 
We thank the anonymous reviewers for their valuable remarks.

\appendix
\section{$L^2$-error estimate for the Galerkin method}
In this appendix, we prove an $L^2$-estimate for the Galerkin method \eqref{eq:linlodgal}.
\begin{theorem}\label{thm:errorgalL2}
	Let Assumptions \ref{asspt:monotone} and \ref{asspt:linearization} be fulfilled and suppose that $A(x, \nabla u)=A_L(x, \nabla u, \nabla u)$.
	Let $u$ be the solution to \eqref{eq:nonlinearweak} and $u_{H,m}$ the solution to \eqref{eq:lodgal}.
	Then it holds that
	\begin{align*}
		\|u-u_{H,m}\|_0&\lesssim \bigl(H+C_{\operatorname{ol}, m}^{1/2}\beta^m+\sup_{v_H\in V_H}|(\CQ_m-\CQ_m^u)v_H|_1\bigr) |u-u_{H,m}|_1+\|\eta_{\mathrm{lin},1}(u, \widetilde{u}_{H,m})\|_{H^{-1}(\Omega)}
	\end{align*}
	with the linearization error $\eta_{\mathrm{lin},1}$ defined as
	\[\langle \eta_{\mathrm{lin},1}(u,v), \psi\rangle:=\int_\Omega\bigl(A_L(x, \nabla u,\nabla v)-A(x, \nabla v)\bigr)\cdot \nabla \psi\,  dx\]
	and
	where $\widetilde{u}_{H, m}$ is defined via \eqref{eq:ellipticproj} below.
\end{theorem}

The first term of the above error estimate presumably is of $O(H^2)$ for the choice $m\approx |\log H|$, see Section~\ref{subsec:errorlod}.
The linearization error $\eta_{\mathrm{lin},1}(u,v)$ can be estimated in the same way as in Lemmas \ref{lem:linKacanov} and \ref{lem:linNewton} by replacing $u^*$ with $u$.
We then need to estimate $|u-\widetilde{u}_{H,m}|_1$ and $\|\nabla(u-\widetilde{u}_{H,m})\|_{L^\infty(\Omega)}$.
As shown below, $|u-\widetilde{u}_{H,m}|_1$ is of the same order as $|u-u_{H,m}|_1$ discussed in Theorem \ref{thm:errorgal}.
However, it is not clear whether $W^{1,\infty}$-estimates for the multiscale space $V_{H,m}$ can be established in the spirit of \cite[Chapter 8]{BrenScott94} to obtain a quadratic rate for $\eta_{\mathrm{lin}, 1}$ in case of a Newton-type linearization.

\begin{proof}
	Let $\widetilde{u}_{H, m}\in V_{H, m}$ be the unique solution to
	\begin{equation}\label{eq:ellipticproj}
		\bigl(\mathfrak{A}^u\nabla \widetilde{u}_{H, m}, \nabla v_{H, m}\bigr)_\Omega=\bigl(\mathfrak{A}^u\nabla u, \nabla v_{H, m}\bigr)_\Omega\qquad \text{for all}\quad v_{H, m}\in V_{H,m},
	\end{equation}
	cf.~\cite{AH16nonlinelliptic}.
	Note that this can equivalently be written as
	\[\bigl(A_L(x, \nabla u, \nabla \widetilde{u}_{H, m}), \nabla v_{H, m}\bigr)_\Omega=\bigl(A_L(x, \nabla u,\nabla u), \nabla v_{H, m}\bigr)_\Omega\qquad \text{for all}\quad v_{H, m}\in V_{H,m}.\]
	We split the error $u-u_{H,m}$ into $u-\widetilde{u}_{H,m}$ and $\widetilde{u}_{H,m}-u_{H,m}$ and estimate both parts separately.
	
	\emph{First step: Estimate of $\widetilde{u}_{H,m}-u_{H,m}$:} By the monotonicity of $A$, Galerkin orthogonality, the assumption $A(x,\nabla u)=A_L(x, \nabla u, \nabla u)$ and \eqref{eq:ellipticproj}, we deduce
	\begin{align*}
		|u_{H,m}-\widetilde{u}_{H,m}|^2_1
		&\lesssim \bigl(A(x, \nabla u_{H, m})-A(x, \nabla \widetilde{u}_{H,m}), \nabla (u_{H,m}-\widetilde{u}_{H, m})\bigr)_\Omega\\
		&=\bigl(A(x, \nabla u)-A(x, \nabla \widetilde{u}_{H,m}), \nabla (u_{H,m}-\widetilde{u}_{H, m})\bigr)_\Omega\\
		&=\bigl(A_L(x, \nabla u, \nabla u)-A(x, \nabla \widetilde{u}_{H,m}), \nabla (u_{H,m}-\widetilde{u}_{H, m})\bigr)_\Omega\\
		&=\bigl(A_L(x, \nabla u, \nabla \widetilde{u}_{H,m})-A(x, \nabla \widetilde{u}_{H,m}), \nabla (u_{H,m}-\widetilde{u}_{H, m})\bigr)_\Omega\\
		&\leq \|\eta_{\mathrm{lin}, 1}(u, \widetilde{u}_{H, m})\|_{H^{-1}(\Omega)}|u_{H,m}-\widetilde{u}_{H,m}|_1.
	\end{align*}
	The $L^2$-estimate is obtained by applying Friedrich's inequality.
	
	\emph{Second step: Estimate of $u-\widetilde{u}_{H,m}$:} This is in principle an $L^2$-estimate for a Galerkin LOD for an elliptic diffusion problem. The main issue, however, is that the multiscale space $V_{H,m}$ is not built with respect to the diffusion tensor $\mathfrak{A}^u=\VA_L(x, \nabla u)$ but with respect to $\mathfrak{A}=\VA_L(x, \nabla u^*)$.
	We first of all note that due to the projection, we have 
	\[|u-\widetilde{u}_{H,m}|_1\lesssim \inf_{v_{H,m}\in V_{H,m}}|u-v_{H,m}|_1\leq |u-u_{H,m}|_1.\]
	Let now $z\in H^1_0(\Omega)$ and $z_{H,m}\in V_{H, m}$ be the solutions of the following dual problems
	\begin{align*}
		\bigl(\mathfrak{A}^u\nabla v, \nabla z\bigr)_\Omega&=(v, u-\widetilde{u}_{H,m})_\Omega\qquad \text{for all}\quad v\in H^1_0(\Omega),\\
		\text{and}\quad \bigl(\mathfrak{A}^u\nabla v_{H,m}, \nabla z_{H,m}\bigr)_\Omega&=(v_{H,m}, u-\widetilde{u}_{H,m})_\Omega\qquad \text{for all}\quad v_{H,m}\in V_{H,m}.
	\end{align*}
	Assumption \ref{asspt:linearization} and Galerkin orthogonality imply
	\begin{align*}
		|z-z_{H,m}|_1&\lesssim \inf_{v_{H,m}\in V_{H,m}}|z-v_{H,m}|_1\leq |z-(\operatorname{id}-\CQ_m)I_Hz|_1\\
		&\leq |z-(\operatorname{id}-\CQ^u)I_H z|_1+|(\CQ^u-\CQ^u_m) I_H z|_1+|(\CQ^u_m-\CQ_m)I_H z|_1,
	\end{align*}
	where the second term is estimated with Proposition \ref{prop:errortrunc}.
	Employing \eqref{eq:IHstab} and the a priori (stability) estimate for $z$, the last term yields
	\[|(\CQ^u_m-\CQ_m)I_H z|_1\lesssim \Bigl(\sup_{v_H\in V_H}|(\CQ^u_m-\CQ_m)v_H|_1\Bigr)\|u-\widetilde{u}_{H,m}\|_0.\]
	For the first term, we obtain with the ellipticity of $\mathfrak{A}^u$ as well as the definitions of $\CQ^u$ and $z$ that
	\begin{align*}
		|z-(\operatorname{id}-\CQ^u) I_H z|_1^2&\lesssim \bigl(\mathfrak{A}^u \nabla(z-(\operatorname{id}-\CQ^u)I_H z), \nabla(z-(\operatorname{id}-\CQ^u)I_H z)\bigr)_\Omega\\
		&=\bigl(\mathfrak{A}^u\nabla z, \nabla(z-(\operatorname{id}-\CQ^u)I_H z)\bigr)_\Omega\\
		&=(u-\widetilde{u}_{H,m}, z-(\operatorname{id}-\CQ^u)I_H z)_\Omega\\
		&\lesssim H\, \|u-\widetilde{u}_{H, m}\|_0\,|z-(\operatorname{id}-\CQ^u)I_H z|_1.
	\end{align*}
	Combining the foregoing estimates, we conclude
	\[|z-z_{H,m}|_1\lesssim \bigl(H+C_{\operatorname{ol}, m}^{1/2}\beta^m+\sup_{v_H\in V_H}|(\CQ_m-\CQ_m^u)v_H|_1\bigr) \|u-\widetilde{u}_{H,m}\|_0.\]
	Finally, the definition of the dual problems yields
	\begin{align*}
		\|u-\widetilde{u}_{H,m}\|_0^2&=\bigl(\mathfrak{A}^u\nabla (u-\widetilde{u}_{H,m}), \nabla z\bigr)_\Omega\\
		&=\bigl(\mathfrak{A}^u\nabla (u-\widetilde{u}_{H,m}), \nabla (z-z_{H,m})\bigr)_\Omega\\
		&\lesssim |z-z_{H,m}|_1|u-\widetilde{u}_{H,m}|_1,
	\end{align*}
	which in combination with the already derived estimates finishes the proof. 
\end{proof}

\section{Proof of Proposition \ref{prop:correcerror}}
\label{app:proofcorrecerror}

The proof of Proposition \ref{prop:correcerror} simply relies on the definition of the element correctors and is similar to the results in \cite{HKM20lodperturbed,HM19lodsimilar}.

\begin{proof}[Proof of Proposition \ref{prop:correcerror}]
	Let $T\in \CT_H$ and $v_H\in V_H$ be fixed.
	Abbreviate $w:=(\CQ_{T,m}-\CQ_{T,m}^u)v_H$ and note that $w\in W(\UN^m(T))$.
	We deduce by the definition of $\CQ_{T,m}$ and $\CQ_{T,m}^u$ that
	\begin{align*}
		|w|_1^2&\lesssim \bigl(\widehat{\mathfrak{A}^u}\nabla w, \nabla w\bigr)_{\UN^m(T)}\\
		&=\bigl(\widehat{\mathfrak{A}^u}\nabla\CQ_{T,m} v_H, \nabla w\bigr)_{\UN^m(T)}-\bigl(\widehat{\mathfrak{A}}\nabla\CQ_{T,m} v_H, \nabla w\bigr)_{\UN^m(T)}+\bigl(\widehat{\mathfrak{A}}\nabla v_H, \nabla w\bigr)_T-\bigl(\widehat{\mathfrak{A}^u}\nabla v_H, \nabla w\bigr)_T\\
		&=\bigl((\widehat{\mathfrak{A}}-\widehat{\mathfrak{A}^u})(\chi_T\nabla -\nabla \CQ_{T,m})v_H, \nabla w\bigr)_{\UN^m(T)}\\
		&\leq \|(\widehat{\mathfrak{A}}-\widehat{\mathfrak{A}^u})(\chi_T\nabla -\nabla \CQ_{T,m})v_H\|_{0,\UN^m(T)}\, |w|_{1}.
	\end{align*}
	We then proceed as follows
	\begin{align*}
		|w|_1^2&\lesssim \|(\widehat{\mathfrak{A}}-\widehat{\mathfrak{A}^u})(\chi_T\nabla -\nabla \CQ_{T,m})v_H\|_{0,\UN^m(T)}^2\\
		&\leq \max_{\psi|_T, \psi\in V_H}\frac{\|(\widehat{\mathfrak{A}}-\widehat{\mathfrak{A}^u})(\chi_T\nabla -\nabla \CQ_{T,m})\psi\|_{0,\UN^m(T)}^2}{|\psi|_{1,T}^2}\, |w|_{1,T}^2\\
		&\leq \sum_{T^\prime\in\CT_H, T^\prime\subset \UN^m(T)}\|\widehat{\mathfrak{A}}-\widehat{\mathfrak{A}^u}\|_{L^\infty(T^\prime)}^2 \max_{\psi|_T, \psi\in V_H}\frac{\|(\chi_T\nabla -\nabla \CQ_{T,m})\psi\|_{0,T^\prime}^2}{|\psi|_{1,T}^2}\, |w|_{1,T}^2\\
		&=E_{\CQ, T}^2\, |w|_{1,T}^2,
	\end{align*}
	which finishes the proof.
	Note that only in the very first step we hide the lower spectral bound of $\widehat{\mathfrak{A}^u}|_{\UN^m(T)}$ in the notation $\lesssim$, all other estimates are constant-free.
\end{proof}


\begin{thebibliography}{10}
	
	\bibitem{ABV15rbhmmquasilinelliptic}
	A.~Abdulle, Y.~Bai, and G.~Vilmart.
	\newblock Reduced basis finite element heterogeneous multiscale method for
	quasilinear elliptic homogenization problems.
	\newblock {\em Discrete Contin. Dyn. Syst. Ser.~S}, 8(1):91--118, 2015.
	
	\bibitem{AH16nonlinelliptic}
	A.~Abdulle and M.~E. Huber.
	\newblock Error estimates for finite element approximations of nonlinear
	monotone elliptic problems with application to numerical homogenization.
	\newblock {\em Numer. Methods Partial Differential Equations}, 32(3):955--969,
	2016.
	
	\bibitem{AH16fehmmnonlinparabolic}
	A.~Abdulle and M.~E. Huber.
	\newblock Finite element heterogeneous multiscale method for nonlinear monotone
	parabolic homogenization problems.
	\newblock {\em ESAIM Math. Model. Numer. Anal.}, 50(6):1659--1697, 2016.
	
	\bibitem{AHV15nonlinparaboliclinearized}
	A.~Abdulle, M.~E. Huber, and G.~Vilmart.
	\newblock Linearized numerical homogenization method for nonlinear monotone
	parabolic multiscale problems.
	\newblock {\em Multiscale Model. Simul.}, 13(3):916--952, 2015.
	
	\bibitem{AV12quasilin}
	A.~Abdulle and G.~Vilmart.
	\newblock A priori error estimates for finite element methods with numerical
	quadrature for nonmonotone nonlinear elliptic problems.
	\newblock {\em Numer. Math.}, 121(3):397--431, 2012.
	
	\bibitem{AV14fehmmquasilinelliptic}
	A.~Abdulle and G.~Vilmart.
	\newblock Analysis of the finite element heterogeneous multiscale method for
	quasilinear elliptic homogenization problems.
	\newblock {\em Math. Comp.}, 83(286):513--536, 2014.
	
	\bibitem{All92twoscale}
	G.~Allaire.
	\newblock Homogenization and two-scale convergence.
	\newblock {\em SIAM J. Math. Anal.}, 23(6):1482--1518, 1992.
	
	\bibitem{BrenScott94}
	S.~C. Brenner and L.~R. Scott.
	\newblock {\em The mathematical theory of finite element methods}, volume~15 of
	{\em Texts in Applied Mathematics}.
	\newblock Springer-Verlag, New York, 1994.
	
	\bibitem{CESY17gmsfemnonlinear}
	E.~Chung, Y.~Efendiev, K.~Shi, and S.~Ye.
	\newblock A multiscale model reduction method for nonlinear monotone elliptic
	equations in heterogeneous media.
	\newblock {\em Netw. Heterog. Media}, 12(4):619--642, 2017.
	
	\bibitem{Ciarlet}
	P.~G. Ciarlet.
	\newblock {\em The finite element method for elliptic problems}, volume~40 of
	{\em Classics in Applied Mathematics}.
	\newblock Society for Industrial and Applied Mathematics (SIAM), Philadelphia,
	PA, 2002.
	
	\bibitem{EGLP14gmsfemnonlin}
	Y.~Efendiev, J.~Galvis, G.~Li, and M.~Presho.
	\newblock Generalized multiscale finite element methods. {N}onlinear elliptic
	equations.
	\newblock {\em Commun. Comput. Phys.}, 15(3):733--755, 2014.
	
	\bibitem{EHG04msfem}
	Y.~Efendiev, T.~Y. Hou, and V.~Ginting.
	\newblock Multiscale finite element methods for nonlinear problems and their
	applications.
	\newblock {\em Commun. Math. Sci.}, 2(4):553--589, 2004.
	
	\bibitem{EEV11nonlinearapost}
	L.~El~Alaoui, A.~Ern, and M.~Vohral\'{\i}k.
	\newblock Guaranteed and robust a posteriori error estimates and balancing
	discretization and linearization errors for monotone nonlinear problems.
	\newblock {\em Comput. Methods Appl. Mech. Engrg.}, 200(37-40):2782--2795,
	2011.
	
	\bibitem{EGH15LODpetrovgalerkin}
	D.~Elfverson, V.~Ginting, and P.~Henning.
	\newblock On multiscale methods in {P}etrov-{G}alerkin formulation.
	\newblock {\em Numer. Math.}, 131(4):643--682, 2015.
	
	\bibitem{EHMP16LODimpl}
	C.~Engwer, P.~Henning, A.~M{\aa}lqvist, and D.~Peterseim.
	\newblock Efficient implementation of the localized orthogonal decomposition
	method.
	\newblock {\em Comput. Methods Appl. Mech. Engrg.}, 350:123--153, 2019.
	
	\bibitem{FZ87femnonlinear}
	M.~Feistauer and A.~\v{Z}en\'{\i}\v{s}ek.
	\newblock Finite element solution of nonlinear elliptic problems.
	\newblock {\em Numer. Math.}, 50(4):451--475, 1987.
	
	\bibitem{GP17lodhom}
	D.~Gallistl and D.~Peterseim.
	\newblock Computation of quasi-local effective diffusion tensors and
	connections to the mathematical theory of homogenization.
	\newblock {\em Multiscale Model. Simul.}, 15(4):1530--1552, 2017.
	
	\bibitem{HKM20lodperturbed}
	F.~Hellman, T.~Keil, and A.~M{\aa}lqvist.
	\newblock Numerical upscaling of perturbed diffusion problems.
	\newblock {\em SIAM J. Sci. Comput.}, 42(4):A2014--A2036, 2020.
	
	\bibitem{HM19lodsimilar}
	F.~Hellman and A.~M{\aa}lqvist.
	\newblock Numerical homogenization of elliptic {PDE}s with similar
	coefficients.
	\newblock {\em Multiscale Model. Simul.}, 17(2):650--674, 2019.
	
	\bibitem{Henn11phdhmm}
	P.~Henning.
	\newblock {\em Heterogeneous multiscale finite element methods for
		advection-diffusion and nonlinear elliptic multiscale problems}.
	\newblock PhD thesis, WWU M\"unster, 2011.
	
	\bibitem{HM14LODbdry}
	P.~Henning and A.~M{\aa}lqvist.
	\newblock Localized orthogonal decomposition techniques for boundary value
	problems.
	\newblock {\em SIAM J. Sci. Comput.}, 36(4):A1609--A1634, 2014.
	
	\bibitem{HMP14lodsemilinear}
	P.~Henning, A.~M{\aa}lqvist, and D.~Peterseim.
	\newblock A localized orthogonal decomposition method for semi-linear elliptic
	problems.
	\newblock {\em ESAIM Math. Model. Numer. Anal.}, 48(5):1331--1349, 2014.
	
	\bibitem{HMP14lodgrosspit}
	P.~Henning, A.~M{\aa}lqvist, and D.~Peterseim.
	\newblock Two-level discretization techniques for ground state computations of
	{B}ose-{E}instein condensates.
	\newblock {\em SIAM J. Numer. Anal.}, 52(4):1525--1550, 2014.
	
	\bibitem{HO15hmmmonotone}
	P.~Henning and M.~Ohlberger.
	\newblock Error control and adaptivity for heterogeneous multiscale
	approximations of nonlinear monotone problems.
	\newblock {\em Discrete Contin. Dyn. Syst. Ser. S}, 8(1):119--150, 2015.
	
	\bibitem{HP13oversampl}
	P.~Henning and D.~Peterseim.
	\newblock Oversampling for the multiscale finite element method.
	\newblock {\em Multiscale Model. Simul.}, 11(4):1149--1175, 2013.
	
	\bibitem{Hoa08sparsefemnonlinear}
	V.~H. Hoang.
	\newblock Sparse finite element method for periodic multiscale nonlinear
	monotone problems.
	\newblock {\em Multiscale Model. Simul.}, 7(3):1042--1072, 2008.
	
	\bibitem{Hub15phdhmm}
	M.~Huber.
	\newblock {\em Numerical homogenization methods for advection-diffusion and
		nonlinear monotone problems with multiple scales}.
	\newblock PhD thesis, EPFL, 2015.
	
	\bibitem{KPY16LODiterative}
	R.~Kornhuber, D.~Peterseim, and H.~Yserentant.
	\newblock An analysis of a class of variational multiscale methods based on
	subspace decomposition.
	\newblock {\em Math. Comp.}, 87(314):2765--2774, 2018.
	
	\bibitem{KY16LODiterative}
	R.~Kornhuber and H.~Yserentant.
	\newblock Numerical homogenization of elliptic multiscale problems by subspace
	decomposition.
	\newblock {\em Multiscale Model. Simul.}, 14(3):1017--1036, 2016.
	
	\bibitem{LNW02twoscale}
	D.~Lukkassen, G.~Nguetseng, and P.~Wall.
	\newblock Two-scale convergence.
	\newblock {\em Int. J. Pure Appl. Math.}, 2(1):35--86, 2002.
	
	\bibitem{MV20lodNLH}
	R.~Maier and B.~Verf{\"u}rth.
	\newblock Multiscale scattering in nonlinear {K}err-type media.
	\newblock arXiv preprint, arXiv:2011.09168, 2020.
	
	\bibitem{MP14LOD}
	A.~M{\aa}lqvist and D.~Peterseim.
	\newblock Localization of elliptic multiscale problems.
	\newblock {\em Math. Comp.}, 83(290):2583--2603, 2014.
	
	\bibitem{MP20lodbook}
	A.~M{\aa}lqvist and D.~Peterseim.
	\newblock {\em Numerical Homogenization by Localized Orthogonal Decomposition}.
	\newblock SIAM Spotlights. Society for Industrial and Applied Mathematics
	(SIAM), Philadelphia, PA, 2020.
	
	\bibitem{Owhadi2017}
	H.~Owhadi.
	\newblock Multigrid with rough coefficients and multiresolution operator
	decomposition from hierarchical information games.
	\newblock {\em SIAM Rev.}, 59(1):99--149, 2017.
	
	\bibitem{OZ11gfem}
	H.~Owhadi and L.~Zhang.
	\newblock Localized bases for finite-dimensional homogenization approximations
	with nonseparated scales and high contrast.
	\newblock {\em Multiscale Model. Simul.}, 9(4):1373--1398, 2011.
	
	\bibitem{Pet15LODreview}
	D.~Peterseim.
	\newblock Variational multiscale stabilization and the exponential decay of
	fine-scale correctors.
	\newblock In {\em Building bridges: connections and challenges in modern
		approaches to numerical partial differential equations}, volume 114 of {\em
		Lect. Notes Comput. Sci. Eng.}, pages 341--367. Springer, Cham, 2016.
	
	\bibitem{PVV19homloddomaindecomp}
	D.~Peterseim, D.~Varga, and B.~Verf\"urth.
	\newblock From domain decomposition to homogenization theory.
	\newblock In {\em Domain Decomposition Methods in Science and Engineering XXV},
	volume 138 of {\em Lect. Notes Comp. Sci. Eng.}, pages 29--40. Springer,
	2020.
	
	\bibitem{Richards}
	L.~A. Richards.
	\newblock Capillary conduction of liquids through porous mediums.
	\newblock {\em Physics}, 1(5):318--333, 1931.
	
	\bibitem{vgen80genuchtenmodel}
	M.~van Genuchten.
	\newblock A closed form equations for predicting the hydraulic conductivity of
	unsaturated soils.
	\newblock {\em Soil Sci. Soc. Am. J.}, 40:892--898, 1980.
	
	\bibitem{Xu96twogridnonlinear}
	J.~Xu.
	\newblock Two-grid discretization techniques for linear and nonlinear {PDE}s.
	\newblock {\em SIAM J. Numer. Anal.}, 33(5):1759--1777, 1996.
	
	\bibitem{Zeidler}
	E.~Zeidler.
	\newblock {\em Nonlinear functional analysis and its applications. {IV}.
		Applications to mathematical physics}.
	\newblock Springer-Verlag, New York, 1988.
	
\end{thebibliography}
\end{document}